\documentclass[10pt,reqno]{amsart}

 \usepackage{amsmath,amsfonts,amssymb,amscd,amsthm,amsbsy,bbm, epsf,calc,enumerate,color,graphicx,verbatim,cite,import}
\usepackage{color}
\usepackage{datetime,appendix}
\usepackage{chapterbib,epstopdf}

\usepackage{float}

\usepackage{ifpdf}
\ifpdf
    \usepackage[colorlinks,citecolor=black,linkcolor=black]{hyperref}
\fi

\theoremstyle{plain}

\newtheorem{thm}{Theorem}[section]

\newtheorem{lem}[thm]{Lemma}

\newtheorem{prop}[thm]{Proposition}

\theoremstyle{definition}

\theoremstyle{remark}                  

\DeclareMathOperator*{\osc}{osc}

\definecolor{darkgreen}{rgb}{0,0.4,0}

\newcommand{\grad}{\nabla}
\newcommand{\Z}{\mathbb{Z}}
\newcommand{\real}{\mathbb{R}}

\def\e{\varepsilon}
\newcommand{\integer}{\mathbb{Z}}

\numberwithin{equation}{section}

\def\Xint#1{\mathchoice
{\XXint\displaystyle\textstyle{#1}}%
{\XXint\textstyle\scriptstyle{#1}}%
{\XXint\scriptstyle\scriptscriptstyle{#1}}%
{\XXint\scriptscriptstyle\scriptscriptstyle{#1}}%
\!\int}
\def\XXint#1#2#3{{\setbox0=\hbox{$#1{#2#3}{\int}$ }
\vcenter{\hbox{$#2#3$ }}\kern-.6\wd0}}

\def\dashint{\Xint-}

\newcommand{\eref}[1]{(\ref{e.#1})}

\begin{document}
\title[Continuity Properties in Boundary Data Homogenization]{Continuity Properties for Divergence Form Boundary Data Homogenization Problems}
\author{William M. Feldman}
\email{feldman@math.uchicago.edu}
\address{Department of Mathematics, The University of Chicago, IL 60637, USA}
\author{Yuming Paul Zhang}
\email{yzhangpaul@math.ucla.edu}
\address{
Department of Mathematics, University of California, Los Angeles 90095, USA}
\begin{abstract}
We study the continuity/discontinuity of the effective boundary condition for periodic homogenization of oscillating Dirichlet data for nonlinear divergence form equations and linear systems.  For linear systems we show continuity, for nonlinear equations we give an example of discontinuity.  
  \end{abstract}
\maketitle

\section{Introduction}

In this work we will study the following type of boundary layer problem in dimension $d \geq 2$
\begin{equation}\label{e.cellmain}
\left\{\begin{aligned}
&- \nabla  \cdot a(y,\grad v_n^s)  = 0 
  &\text{ in }& P_n^s = \{ y \cdot n >s\} ,\\
& v_n^s(y)=\varphi(y)  &\text{ on }& \partial P_n^s.
\end{aligned}
\right.
\end{equation}
Here $n \in S^{d-1}$ is a unit vector, $s \in \real$, $\varphi$ is continuous and $\mathbb{Z}^d$ periodic, the operator $a$ is also $\mathbb{Z}^d$ periodic in $y$ and will satisfy a uniform ellipticity assumption.  In this work will consider both the case of nonlinear scalar equations and linear systems so, for now, we do not specify the assumptions on $a$ any further.  

The \emph{boundary layer limit} of the system \eref{cellmain} is defined by 
\[ \varphi_*(n,s) := \lim_{R \to \infty} v_n(Rn+y) \hbox{ if the limit exists and is independent of $y \in \partial P_n^s$.}\]
 Typically $\varphi_*$ is independent of $s$ for irrational directions $n$ and we write $\varphi_*(n)$, while for rational directions $n \in \real \integer^d$ the limits above exist but depend on $s$.  If, additionally, the boundary layer limit is independent of $s$ then we say that the cell problem \eref{cellmain} homogenizes.  The focus of this article is on the continuity (or lack thereof) of $\varphi_*$ as a function of the normal direction.  The continuity of $\varphi_*$ is intrinsically linked with linearity of the operator $a(x,p)$.  In the case of a linear system we show continuity of $\varphi_*$, while in the case of nonlinear scalar equations we give an example where $\varphi_*$ is discontinuous, this indicates generic discontinuity for nonlinear equations.

Before we continue with our discussion of the continuity properties of $\varphi_*$ we give a brief explanation about where the problem \eref{cellmain} arises and why one should be interested in the continuity of $\varphi_*$.  Let $\Omega\subset \mathbb{R}^d$ be a bounded domain and consider the homogenization problem with oscillating Dirichlet boundary data,
\begin{equation}\label{e.maineqn}
\left\{\begin{aligned}
&- \nabla  \cdot (a(\tfrac{x}{\e},\grad u^\e))  =0 
  &\text{ in }& \Omega ,\\
& u^\e(x)=g(x,\tfrac{x}{\e})  &\text{ on }& \partial\Omega
\end{aligned}
\right.
\end{equation}
where $\e>0$ is a small parameter, $g(x,y)$ is continuous in $x,y$ and $\mathbb{Z}^d$ periodic in $y$.  This system is natural to consider in its own right, but also it arises naturally in the study of homogenization with non-oscillatory Dirichlet data when one studies the higher order terms in the asymptotic expansion, see \cite{Gerard-Varet:2012aa} where this is explained.  

The interest in studying \eref{maineqn} is the asymptotic behavior of the $u^\e$ solutions as $\e \to 0$.  This problem has been studied recently by a number of authors starting with the work of Gerard-Varet and Masmoudi \cite{Gerard-Varet:2012aa,Gerard-Varet:2011aa} and followed by \cite{Aleksanyan:2015aa,Aleksanyan:2017aa, Choi:2014aa,Feldman:2015aa,Feldman:2014aa,Prange:2013aa,Zhang:2017aa, Armstrong:2017aa, Guillen:2016aa}.  It has been established that solutions $u^\e$ converge, at least in $L^2(\Omega)$, to some $u^0$ which is a unique solution to
\begin{equation}
\left\{\begin{aligned}
&-\nabla  \cdot (a^0( \nabla  u^0))  =0
  &\text{ in }&  \Omega,\\
& u^0(x)=\varphi^0(x) &\text{ on }&  \partial\Omega 
\end{aligned}
\right.
\end{equation}
where $a^0$ and $\varphi^0(x)$ are called respectively the homogenized operator and homogenized boundary data.  The identification of the homogenized operator $a^0$ is a classical topic.  The homogenized boundary $\varphi^0$ is determined by the boundary layer problem \eref{cellmain},
\[ \varphi^0(x) = \varphi_*(n_x) \ \hbox{ when $n_x$ is the inward unit normal to $\Omega$ and } \ \varphi(y) = g(x,y).\] 
 That is, \eref{cellmain} can be viewed as a kind of cell problem associated with the homogenization of \eref{maineqn}. At least for linear equations this definition makes sense as long as the set of boundary points of $\partial \Omega$ where \eref{cellmain} does not homogenize, i.e. those with rational normal, has zero harmonic measure.  The convergence of $u^\e$ to $u^0$ has been established rigorously for linear systems by G\'{e}rard-Varet and Masmoudi~\cite{Gerard-Varet:2012aa}, and further investigations have yielded optimal rates of convergence, see Armstrong, Kuusi, Mourrat and Prange \cite{Armstrong:2017aa} and Shen and Zhuge \cite{Shen:2016aa}.  For nonlinear divergence form equations, to our knowledge, the problem has not been studied yet.  This is the source of our interest in the continuity properties of $\varphi_*$.
 
 The main result which we establish in this paper is that the directional limits of $\varphi_*$ at a rational direction are determined by a ``second cell problem" for the homogenized operator $a^0$.  Let us take $\xi \in \integer^d \setminus \{0\}$ to be an irreducible lattice vector and $\hat\xi$ to be the corresponding rational unit vector in the same direction.  Then the cell problem \eref{cellmain} solution $v_{\xi}^s$ exists for each $s \in \real$ and has a boundary layer limit, 
 \[ \varphi_*(\xi,s) := \lim_{R \to \infty} v_{\xi,s}(R\xi),\]
 but that limit typically is not independent of the translation $s$ applied to the half-space domain $P_\xi$.  We will see that $\varphi_*(\xi,s)$ is a $1/|\xi|$-periodic function on $\real$.  Now suppose that we have a sequence of directions $n_k \to \hat\xi$ such that,
 \[ \frac{\hat\xi-n_k}{|\hat \xi-n_k |} \to \eta  \ \hbox{ a unit vector with } \ \eta \perp \xi.\]
Call $\eta$ the approach direction of the sequence $n_k$ to $\xi$.  We will show that the limit of $\varphi_*(n_k)$ is determined by the following boundary layer problem. Define
 \begin{equation}\label{e.cell2}
\left\{
\begin{array}{ll}
 - \grad \cdot a^0(\grad w_{\xi,\eta}) = 0 & \hbox{ in } P_\xi \vspace{1.5mm}\\
 w_{\xi,\eta} = \varphi_*(\xi,x \cdot \eta) & \hbox{ on } \partial P_\xi
\end{array}
\right.
\end{equation}
then it holds
\[ \lim_{k \to \infty} \varphi_*(n_k) = \lim_{R \to \infty} w_{\xi,\eta}(R\xi).\]
Thus the directional limits of $\varphi_*$ at $\xi$ are determined by the boundary layer limit of a half-space problem for the homogenized operator.  This limit structure was first observed in Choi and Kim~\cite{Choi:2014aa} and developed further by the first author and Kim~\cite{Feldman:2015aa}, both papers studied non-divergence form and possibly nonlinear equations. We will explain in this paper how the second cell problem follows purely from \emph{qualitative} features which are shared by a wide class of elliptic equations, including divergence form linear systems, and both divergence and non-divergence form nonlinear equations.  We are somewhat vague about the hypotheses, which will be explained in detail in Sections \ref{sec: linear background} and \ref{sec: nonlinear background}.

\begin{thm}\label{main1}
The limit characterization \eref{cell2} holds for divergence form linear systems and nonlinear equations satisfying a uniform ellipticity condition.
\end{thm}

Once we have established \eref{cell2}, the question of qualitative continuity/discontinuity of $\varphi_*$ is reduced to a much simpler problem.  For linear equations the homogenized operator $a^0$ is linear and translation invariant and so a straightforward argument, for example by the Riesz representation theorem, shows that,
\[ \lim_{R \to \infty} w_{\xi,\eta}(R\xi) = |\xi| \int_0^{1/|\xi|} \varphi_*(\xi,s) \ ds,\]
i.e. it is the average over a period of $\varphi_*(\xi,\cdot)$. Evidently this does not depend on the approach direction $\eta$.  Thus qualitative continuity of $\varphi_*$ for linear problems follows easily once we establish \eref{cell2}.  Our arguments to derive \eref{cell2} can be quantified to obtain a modulus of continuity, which we make explicit below, however so far we cannot push the method to obtain the optimal modulus of continuity. The recent work Shen and Zhuge~\cite{Shen:2017aa} obtains an almost Lipschitz modulus of continuity by a different method, we will compare their approach with ours below. 

\begin{thm}\label{main2}
For elliptic linear systems, $d \geq 2$, for any $0 < \alpha < 1/d$ there is a constant $C\geq1$ depending on $\alpha$ as well as universal parameters associated with the system (see Section~\ref{sec: linear background}) such that, for any $n_1,n_2$ irrational,
\[  |\varphi_*(n_1) - \varphi_*(n_2) | \leq C \|\varphi\|_{C^5} |n_1 - n_2|^{\alpha}.\]
\end{thm}

For nonlinear problems we expect that $\varphi_*$ is discontinuous at rational directions, at least for generic boundary data and operators.  A result of this kind was established for non-divergence form equations in \cite{Feldman:2015aa}.  We have not been able to prove such a general result for divergence form nonlinear equations, but we have constructed an explicit example showing that discontinuity is possible.

\begin{thm}\label{main3}
For $d \geq 3$ there exist smooth boundary data $\varphi$ and uniformly elliptic, positively $1$-homogeneous, nonlinear operators $a(x,p)$ such that $\varphi_*$ is discontinuous at some rational direction.
\end{thm}

 We compare with the work of Shen and Zhuge~\cite{Shen:2017aa} which studies continuity properties of $\varphi_*$ for linear divergence form systems.  They show, in the linear systems case, that $\varphi_*$ is in $W^{1,p}$ for every $p<\infty$.  Their work relies on a formula for the homogenized boundary condition from Armstrong, Kuusi, Mourrat and Prange \cite{Armstrong:2017aa} which in turn relies on the result of Kenig, Lin and Shen~\cite{Kenig:2014aa} giving precise asymptotics for the Green's function and Poisson kernel associate with the equation \eref{maineqn}. For this reason their result seems to be restricted to the linear setting.  Our approach, although it does not yield an optimal quantitative estimate as of yet, does not rely on formulas for the Poisson kernel asymptotics and applies to both linear and nonlinear equations (including both divergence form, as established here, and nondivergence form, as established previously in \cite{Feldman:2015aa}).  We note that in the course of proving Theorem~\ref{main2} we actually show H\"{o}lder regularity for every $0 < \alpha < 1$ at each lattice direction $\xi \in \integer^d \setminus \{0\}$, the modulus of continuity however depends on the rational direction and degenerates as $|\xi| \to \infty$.  This is why we only end up with (almost) H\"{o}lder-$\frac{1}{d}$ continuity in the end.

 \subsection{Notation} We go over some of the notations and terminology used in the paper.  We will refer to constants which depend only on the dimension or fundamental parameters associated with the operator $a(x,p)$ (to be made specific below), e.g. ellipticity ratio or smoothness norm, as universal constants.  We will write $C$ or $c$ for universal constants which may change from line to line.  Given some quantities $A,B$ we write $A \lesssim B$ if $A \leq CB$ for a universal constant $C$.  If the constants depend on an additional non-universal parameter $\alpha$ we may write $A \lesssim_\alpha B$.  
 
 We will use various standard $L^p$ and H\"{o}lder $C^{k,\alpha}$ norms.  For H\"{o}lder semi-norms, which omit the zeroth order $\sup$ norm term, we write $[f]_{C^{k,\alpha}}$.  Given a measurable set $E \subset \real^d$ we will also use the $L^p_{avg}(E)$ norm which is defined by
 \[ \|f\|_{L^p_{avg}(E)} = \left(\frac{1}{|E|}\int_{E} |f|^p \right)^{1/p}.\]
 The oscillation is a convenient quantity for us since the solution property for the equations we consider is preserved under addition of constant functions. This is usually defined for a scalar valued function $u : E \to \real$ on a set $E \subset \real^d$ as $\osc_E u=\sup_E u - \inf_E u$.  We use a slightly different definition which also makes sense for vector valued $u : E \to \real^N$,
 \[ \osc_E u := \inf \{ r >0 : \ \hbox{ there exists }  u_0 \in \real^N  \hbox{ s.t. }   \|u - u_0\|_{L^\infty(E)} \leq r/2 \}.\]

 \section{Explanation of the limit structure at rational directions}\label{sec: outline}
   We give a high level description of the asymptotics of the boundary layer limit at rational directions.  What we would like to emphasize throughout this description is that the argument is basically geometric, and has to do with the way that $\partial P_n$ intersects the unit periodicity cell in the asymptotic limit as $n$ approaches a rational direction.  This calculation relies only on certain qualitative features of Dirichlet problems for elliptic equations which are true both for divergence and non-divergence form both linear (including systems) and non-linear.  To emphasize the level of abstraction we will write the boundary layer problem in the following form
 \begin{equation}\label{e.abstractcell}
 \left\{
 \begin{array}{ll}
 F[v_n,x] = 0 & \hbox{ in } P_n := \{ x \cdot n >0\} \vspace{1.5mm}\\
 v_n = \varphi & \hbox{ on } \partial P_n
 \end{array}\right.
 \end{equation}
 Always $F$ and $\varphi$ will share $\integer^d$ periodicity in the $x$ variable.  In order to carry out the heuristic argument we will need the following properties of the class of equations/systems.  We emphasize that the following properties are not stated very precisely, they are merely meant to be illustrative.
\begin{enumerate}[$(i)$]
\item (Homogenization) There is an elliptic operator $F^0$ in the same class such that if $u^\e$ is a sequence of solutions of $F[u^\e,\tfrac{x}{\e}] = 0$ in a domain $\Omega$ converging to some $u^0$ then $F[u^0] = 0$ in $\Omega$.
\item (Continuity with respect to boundary data in $L^\infty$)  There exists $C>0$ so that if $n \in S^{d-1}$ and $u_1$, $u_2$ are bounded solutions of \eref{abstractcell} with respective boundary data $\varphi_1$ and $\varphi_2$ then,
\[ \sup_{P_n} |u_1 - u_2| \leq C \sup_{\partial P_n} |\varphi_1 - \varphi_2|.\]
\item (Large scale interior and boundary regularity estimates) There is $\alpha \in (0,1)$ such that for any $r>0$ if $F[u,x] = 0$ in $B_r \cap P_n$, where $B_r$ is some ball of radius $r$, 
\[ [u]_{C^{\alpha}(B_{r/2} \cap P_n)} \lesssim  r^{-\alpha}\osc_{B_{r} \cap P_n} u + [g]_{C^\alpha( B_{r} \cap \partial P_n)}. \]
 \end{enumerate}

  The heuristic outline below applies to a wide class of elliptic equations, already the arguments were carried out rigorously for non-divergence nonlinear equations by Choi and Kim~\cite{Choi:2014aa} and the first author and Kim~\cite{Feldman:2015aa} and similar ideas were used for parabolic equations in moving domains by the second author in~\cite{Zhang:2017aa}.  Here we will be studying divergence form equations, linear systems and nonlinear scalar equations. 
 
 To begin we need to understand the boundary layer limit at a rational direction.  Let $\xi \in \integer^d \setminus \{0\}$ and consider the solution $v^s_\xi(x)$ of,
  \begin{equation}\label{e.abstractcellrat}
 \left\{
 \begin{array}{ll}
 F[v^s_\xi,x] = 0 & \hbox{ in } P_\xi^s = \{ x \cdot n > s\} \vspace{1.5mm}\\
 v^s_\xi = \varphi & \hbox{ on } \partial P_\xi^s.
 \end{array}\right.
 \end{equation}
 Translating the half-space, by changing $s$, changes the part of the data $\varphi$ seen by the boundary condition.  Thus the boundary layer limit of $v^s_\xi$ can depend on the parameter $s$, we define
 \[ \varphi_*(\xi,s) = \lim_{R \to \infty} v^s_\xi(R\xi).\]
 As will become clear, this particular parametrization of the boundary layer limits is naturally associated with the asymptotic structure of the boundary layer limits for directions $n$ near $\xi$.
 
 The next step is to understand the geometry near $\xi$.  Let $n \in S^{d-1}$ be a direction near $\xi$ and $v_n$ be the corresponding half-space solution.  We can write,
 \[ n = (\cos \e)\hat\xi -(\sin \e) \eta \ \hbox{ for some small angle $\e$ and a unit vector } \ \eta \perp \xi.\]
We obtain an asymptotic for $v_n$ at an intermediate length scale. 
 
 Let $x \in \partial P_n$, since $n$ is close to $\xi$, $\partial P_n$ is close to $\partial P_\xi + x \cdot \hat \xi$ in a large neighborhood, any scale $o(1/\e)$, of $x$.  By using the local up to the boundary regularity we see that $v_n$ and $v^{s}_\xi$, with $s = x \cdot \hat \xi$, are close on the boundary of their common domain, at least in this $o(1/\e)$ neighborhood of $x$. Now $v_\xi^s$ has a boundary layer limit $\varphi_*(\xi,s)$, and the length scale $|\xi|$ associated with the boundary layer depends on $\xi$, but not on $\e$.  Thus for $\e$ small and $ |\xi| \ll R \ll 1/\e$
 \[ v_n(x+Rn) = \varphi_*(\xi,x \cdot \hat \xi) + o_\e(1) = \varphi_*(\xi,\tan \e (x \cdot \eta)) + o_\e(1).\]
This is one of the main places where we use the large scale boundary regularity estimates, property $(iii)$ above. Thus, moving into the domain by $Rn$ and rescaling to the scale $1/\tan \e$, i.e. letting $w^\epsilon(x)\sim v_n(\frac{x+Rn}{\tan \e})$, we find that the boundary layer limit is well approximated by the boundary layer limit of
   \begin{equation}\label{e.cell2ep}
 \left\{
 \begin{array}{ll}
 F[w^\e,\tfrac{x}{\tan\e}] = 0 & \hbox{ in } P_\xi  \vspace{1.5mm}\\
 w^\e = \varphi_*(\xi,x \cdot \eta) & \hbox{ on } \partial P_\xi
 \end{array}\right.
 \end{equation}
 in the limit as $\e \to 0$.  Now taking the limit as $\e \to 0$ of in \eref{cell2ep} we find the ``second cell problem"
   \begin{equation}\label{e.cell2F}
 \left\{
 \begin{array}{ll}
F^0[w_{\xi,\eta}] = 0 & \hbox{ in } P_\xi  \vspace{1.5mm}\\
 w_{\xi,\eta} = \varphi_*(\xi,x \cdot \eta) & \hbox{ on } \partial P_\xi.
 \end{array}\right.
 \end{equation}
   Thus we characterize the directional limits at the rational direction $\xi$ as the boundary layer limits of the associated second cell problem
 \[ \lim_{k \to \infty} \varphi_*(n_k) = \lim_{R \to \infty} w_{\xi,\eta}(R\xi) \ \hbox{ if } \ \frac{\hat\xi -n_k}{| \hat\xi-n_k|} \to \eta.\]
 With this characterization the \emph{qualitative} continuity and discontinuity of $\varphi_*$ can be investigated solely by studying the problem \eref{cell2F}.  
 
 In the following, Section~\ref{sec: linear background} and Section~\ref{sec: nonlinear background}, we will explain background regularity results for linear systems and nonlinear divergence form equations and the well-posedness of Dirichlet problems in half-spaces.  In particular we will prove that properties we used in the heuristic arguments above do hold for the type of equations/systems we consider.  In Section~\ref{sec: boundary layers} we will go into more detail about the boundary layer problem \eref{cellmain} in rational and irrational half-spaces.  In Section~\ref{sec: asymptotics} we will make rigorous the above outline obtaining intermediate scale asymptotics which lead to the second cell problem \eref{cell2F}.  In Section~\ref{sec: continuity} we show how to derive continuity of $\varphi_*$ from the second cell problem for linear problems, and in Section~\ref{sec: discontinuity} we show how nonlinearity can cause discontinuity of $\varphi_*$.

\section{Linear Systems Background Results}\label{sec: linear background}
 In this section we will recall some results about divergence form linear systems.   Let $\Omega$ be a domain of $\real^d$ and $N \geq 1$, we consider solutions of the following elliptic linear system:
\[ - \nabla  \cdot (A(x) \nabla  u) =0  \ \hbox{ in } \ \Omega\]
where $u \in H^1(\Omega;\real^N)$ is at least a weak solution. Here we use the notation $A = (A_{ij}^{\alpha\beta}(x))$ for $1 \leq \alpha,\beta \leq d$ and $1 \leq i,j \leq N$ defined for $x \in \mathbb{R}^d$, where we mean, using summation convention, 
\[\left(\nabla  \cdot (A(x) \nabla  u^\e)\right)_i = \partial_{x_\alpha}(A^{\alpha\beta}_{ij}(x)\partial_{x_\beta} u_j^\e).\] 
We assume that $A$ satisfies the following hypotheses:
\begin{enumerate}[(i)]
\item Periodicity:
\begin{equation}
A(x+z) = A(x) \ \hbox{ for all } \ x \in \mathbb{R}^d, z \in \mathbb{Z}^d.
\end{equation}
\item Ellipticity: for some $\lambda>0$ and all $\xi \in \mathbb{R}^{d \times N}$,
\begin{equation}
 \lambda \xi^i_\alpha\xi^i_\alpha \leq A^{\alpha\beta}_{ij}\xi^i_\alpha\xi^j_\beta \leq \xi^i_\alpha\xi^i_\alpha.
 \end{equation}
 \item Regularity: for some $M>0$,
 \begin{equation}
 \|A\|_{C^{5}(\mathbb{R}^d)} \leq M.
  \end{equation}
\end{enumerate}
We remark that the regularity on $A$ is far more than is necessary for most of the results below.  When we say that $C$ is a universal constant below we mean that it depends only on the parameters, $d,N,\lambda,M$.

\subsection{Integral Representation}
  Consider the following boundary layer problem, which will be the main object of our study,
\begin{equation}\label{e.linear hs}
\left\{\begin{aligned}
&  -\nabla  \cdot (A(x) \grad u) = \grad \cdot f +g &\text{ in }&  P_n,\\
&  u(x)=\varphi(x) &\text{ on }&   \partial P_n.
\end{aligned}
\right.
\end{equation}
for $f,g$ smooth vector valued functions with compact support and $\varphi$ continuous and bounded. A solution is given by the Green's function formula
\[u(x)=\int_{P_n} \grad G(x,y)\cdot f(y)dy  + \int_{P_n}G(x,y)g(y)dy +\int_{ \partial P_n}P(x,y)\varphi(y)dy.\]
Here $G,P$ are Green matrix and Poisson kernel corresponding to our operator.  For $y \in P_n$, $G$ solves
\begin{equation}
\left\{\begin{aligned}
&  -\nabla_x  \cdot (A(x) \grad_x G(x,y)) = \delta(x-y)I_N &\text{ in }&  P_n,\\
&  G(x,y)=0 &\text{ on }&   \partial P_n
\end{aligned}
\right.
\end{equation}
and the Poisson kernel is given, for $x\in P_n$ and $y \in \partial P_n$, by
\[P(x,y) =  - n \cdot (A^t(y) \grad_y G(x,y)) \ \hbox{ i.e. } \ P_{ij}(x,y) =  - n_\alpha A^{\beta \alpha}_{ki}(y) \partial_{y_\beta}G_{kj}(x,y).\]

Following from the work of Avellaneda-Lin~\cite{Avellaneda:1991aa}, and exactly stated in \cite{Gerard-Varet:2012aa} (Proposition 5), $G$ and $P$ satisfy the same bounds as for a constant coefficient operator:
\begin{thm}\label{integral}
Call $\delta(y):= \textup{dist}(y,\partial P_n)$. For all $x\ne y$ in $ P_n$, one has
\begin{align*}
 |G(x,y)|&\leq\frac{C}{|x-y|^{d-2}}  \quad \text{ for }d\geq 3,\\
|G(x,y)|&\leq {C}(|\log |x-y||+1) \quad \text{ for }d=2,\\
|G(x,y)|&\leq \frac{C\delta(x)\delta(y)}{|x-y|^{d}}  \quad \text{ for all }d,\\
 |\nabla_x G(x,y)|  &\leq \frac{C}{|x-y|^{d-1}} \quad \text{ for all }d,\\
|\nabla_x G(x,y)|  &\leq C(\frac{\delta(y)}{|x-y|^{d}}+
 \frac{\delta(x)\delta(y)}{|x-y|^{d+1}} )
 \quad \text{ for all }d.
\end{align*}
For all $x\in P_n$ and $y\in \partial P_n $, one has
\begin{eqnarray*}
 &&|P(x,y)|\leq\frac{C\delta (x)}{|x-y|^{d}},\\
 &&|\nabla P(x,y)|\leq C(\frac{1}{|x-y|^{d}}+
 \frac{\delta(x)}{|x-y|^{d+1}} ).
\end{eqnarray*}
\end{thm}
Although it is not precisely stated there, the methods of Avellaneda-Lin~\cite{Avellaneda:1991aa} also can achieve the same bounds for the Green's function and Poisson kernel associated with the operator $- \grad \cdot( A(x) \grad )$ in the strip type domains
\[ \Pi_n(0,R) := \{ 0 < x \cdot n < R \}, \]
with constants independent of $R$.  This will be useful later.

From the Poisson kernel bounds we can derive the $L^\infty$ estimate which replaces the maximum principle for linear systems.
\begin{lem}\label{lem: system max}
Suppose that $u_1,u_2$ are bounded solutions of \eref{linear hs} with respective boundary data $\varphi_1,\varphi_2$ and zero right hand side.  Then,
\[  \sup_{P_n} |u_1 - u_2| \leq C\| \varphi_1 - \varphi_2 \|_{L^\infty(\partial P_n)}\]
where $C$ is a universal constant.  The same holds for solutions in $\Pi_n(0,R)$.
\end{lem}

For the solutions given by the Poisson kernel representation formula the result of Lemma~\ref{lem: system max} follows from a standard calculation using Theorem~\ref{integral}.  There is some subtlety in showing uniqueness, see \cite{Gerard-Varet:2012aa} (Section 2.2) for a proof.

\subsection{Large scale boundary regularity}  In this section we consider the large scale boundary regularity used in the heuristic argument of Section~\ref{sec: outline} for linear elliptic systems.  We will need a boundary regularity result from Avellaneda-Lin~\cite{Avellaneda:1987aa} (Theorem 1).  For the below we assume $\Omega$ is some domain with $0 \in \partial \Omega$ and that $u^\e$ solves
\[ - \grad \cdot (A(\tfrac{x}{\e}) \grad u^\e ) = 0 \ \hbox{ in } \ \Omega \cap B_1 \ \hbox{ and } \ u^\e = g \ \hbox{ on } \  \partial \Omega \cap B_1. \]

\begin{lem}\label{lem: flat reg}
For every $0<\alpha <1$ there is a constant $C$ depending on $\alpha$ and universal quantities such that, if $\Omega = \{x_d >0 \} \cap B_1 =: B_1^+$,
\[ [u^\e]_{C^\alpha(B_{1/2}^+)} \leq C( \| \grad g\|_{L^\infty(\{x_d = 0\} \cap B_1)} + \|u^\e - g(0)\|_{L^2(B_1^+)}) ,\]
and for every $\nu >0$
\[ \|\grad u^\e\|_{L^\infty(B_{1/2}^+)} \leq C( \| \grad g\|_{C^{0,\nu}(\{x_d = 0\} \cap B_1)} + \|u^\e - g(0)\|_{L^2(B_1^+)}). \]
\end{lem}

We need the H\"{o}lder regularity result in cone type domains which are the intersection of two half-spaces with normal directions $n_1,n_2$ very close to each other.  We will consider the more general class of domains $\Omega$ which are a Lipschitz graph over $\real^{d-1}$ with small Lipschitz constant. In particular we assume that there is an $f : \real^{d-1} \to \real$ Lipschitz with $f(0) = 0$ such that,
\[ \Omega \cap B_1 = \{(x',x_d): x_N > f(x') \} \cap B_1. \]

\begin{lem}\label{lem: bdry cont linear}
For every $0<\alpha <1$ there is a $\delta(\alpha)>0$ universal such that, if $\Omega$ as above with $\|\grad f\|_{\infty} \leq \delta$, then
\[ [u^\e]_{C^\alpha(\Omega \cap B_{1/2})} \leq C( \| \grad g\|_{L^\infty(\partial \Omega \cap B_1)} + \|u^\e - g(0)\|_{L^2(\Omega \cap B_1)}) .\]
\end{lem}
 The proof is by compactness, we postpone it to Appendix~\ref{sec: A}.
 
 \subsection{Poisson kernel in half-space intersection} From the regularity estimates of the previous subsection we can derive estimates on the Poisson Kernel in the intersection of nearby half-space domains.  Consider two unit vectors $n_1,n_2$ with $|n_1 - n_2| \sim \e$ small.  For simplicity we suppose that,
\[ n_j = (\cos \e) e_d +(-1)^{j}(\sin \e) e_1.\]  
Call
\[ K= P_{n_1} \cap P_{n_2}.\]
Call $G_K(x,y)$ to be the Green's matrix.  Although the domain is Lipschitz $G_K$ still satisfies the bound (via Avellaneda-Lin~\cite{Avellaneda:1987aa}), in $ d \geq 3$,
\[ |G_K(x,y)| \lesssim \frac{1}{|x-y|^{d-2}}.\]
We call $P_K(x,y)$, for $x \in K$ and $y \in \partial K$, to be the Poisson kernel for $K$, which is well-defined as long as $y_1 \neq 0$.  Call $\delta(x)= \textup{dist}(x,\partial K)$. 
\begin{lem}\label{lem: PK bounds 1}
For any $\alpha \in (0,1)$ and $\e$ sufficiently small depending on $\alpha$ and universal quantities,
\[ 
|P_K(x,y)| \lesssim_\alpha
\left\{
\begin{array}{lll}
\frac{\delta(x)^\alpha}{|x-y|^{d-1+\alpha}} & \hbox{ for } & |y_1| \geq \tfrac{1}{2}|x-y| \vspace{1.5mm}\\
\frac{1}{|y_1|}\frac{\delta(x)^\alpha}{|x-y|^{d-2+\alpha}} & \hbox{ for } & |y_1| \leq \tfrac{1}{2}|x-y|.
\end{array}\right.
\]
\end{lem}
The proof is postponed to Appendix~\ref{sec: A}, we show how the estimates are used. Suppose $\psi : \partial K \to \real^N$ satisfies,
\[ |\psi(x)| \leq \min\{|x_1|,1\}.\]
We consider the Poisson kernel solution of the Dirichlet problem,
\[ u(x) = \int_{\partial K} P_K(x,y) \psi(y) \ dy.\]
In particular we are interested in the continuity at $0$, we only consider really $x = te_d$ for some $t>0$ (or $x = tn_1$ or $tn_2$ but this is basically the same) so we restrict to that case.  Now for $y \in \partial K$, $|x-y| \sim t + |y|$ and so $|x-y| \gtrsim |y_1|$ and the first bound in Lemma~\ref{lem: PK bounds 1} implies the second. Thus we can compute
\begin{align*} 
|u(te_d)| &\lesssim \int_{\partial K} \frac{1}{|y_1|}\frac{t^\alpha}{(t+|y|)^{d-2+\alpha}}\min\{|y_1|,1\} \ dy \\
&\lesssim \int_{ \partial K }\frac{t^\alpha}{(t+|y|)^{d-2+\alpha}}\min\{1,\frac{1}{|y_1|}\} \ dy \\
&\lesssim  \int_{\real}\int_{ \real^{d-2} }\min\{1,\frac{1}{|y_1|}\}\frac{t^\alpha}{(t+|y_1| +|z|)^{d-2+\alpha}} \ dzdy_1 
\end{align*}
Computing the inner integrals
\[\int_{ \real^{d-2} }\frac{1}{(t+|y_1| +|z|)^{d-2+\alpha}} \ dz = \frac{1}{(t+|y_1|)^{\alpha}} \int_{\real^{d-2} }\frac{1}{(1+|w|)^{d-2+\alpha}} \ dw \lesssim  \frac{1}{(t+|y_1|)^{\alpha}}. \]
Then
\[ u(te_d) \lesssim \int_{\real} \min\{1,\frac{1}{|y_1|}\}\frac{t^{\alpha}}{(t+|y_1|)^{\alpha}} dy_1 \lesssim t^\alpha \ \hbox{ for } \ t \leq 1. \]
We state the result of a slight generalization of this calculation as a Lemma.
\begin{lem}\label{lem: PK bounds 2}
Suppose that $K = P_{n_1} \cap P_{n_2}$, $\alpha \in (0,1)$ and $\e = |n_1 - n_2|$ is sufficiently small so that the estimates of Lemma~\ref{lem: PK bounds 1} hold, $\psi : \partial K \to \real$ smooth and satisfies the bound $|\psi(x)| \leq \min\{ \delta^\beta |x \cdot (n_1 - n_2)|^\beta,1\}$ for some $\delta >0$ and $1 \geq \beta >\alpha$,  Then for any bounded solution $u$ of
\[ - \grad \cdot( A(x) \grad u) = 0 \ \hbox{ in } \ K \ \ \hbox{ with } \ u = \psi \ \hbox{ on } \ \partial K. \]
it holds
\[ |u(te_d)| \lesssim \delta^\alpha t^\alpha \ \hbox{ for } \ t \leq 1/\delta.\]
\end{lem}
There is an additional subtlety which is the uniqueness of the bounded solution of the Dirichlet problem in $K$, the argument is the same as in the half-space case, see \cite{Gerard-Varet:2012aa}.  To derive Lemma~\ref{lem: PK bounds 2} from the previous calculation just do a rescaling to $u(\cdot/\delta)$, the domain $K$ is scaling invariant and the Poisson kernel associated with $A(\cdot/\delta)$ satisfies the same bounds as for $A$.

\section{Nonlinear Equations Background Results}\label{sec: nonlinear background}
In this section we consider the boundary layer problem for nonlinear operators.  To explain the assumptions we write out the problem in a general domain
\begin{equation}\label{e.nonlinearhomgen}
\left\{\begin{array}{ll}
- \nabla  \cdot a(\tfrac{x}{\e},\grad u^\e)  =0 
  &\text{ in } \Omega ,\vspace{1.5mm}\\
 u^\e(x)=g(x,\frac{x}{\e})  &\text{ on } \partial\Omega.
\end{array}
\right.
\end{equation}
This type of equation would arise as the Euler-Lagrange equation of a variational problem,
\[ \hbox{ minimize } \ E(u) = \int_{\Omega} F(\tfrac{x}{\e},Du) \ dx \ \hbox{ over } \ u \in H^1_0(\Omega) + g(\cdot,\tfrac{\cdot}{\e}).\]
 A natural uniform ellipticity assumption on the functional $F$ is
\[ \hbox{$F$ is convex with $1 \geq D^2F \geq \lambda>0$}.\]
Then $a = DF$ is $1$-Lipschitz continuous in $p$ and has the monotonicity property
\[   (a(x,p) - a(x,q))\cdot(p-q) \geq \lambda|p-q|^2 \ \hbox{ for all } \ p,q \in \real^d. \]

Now we consider how to determine the effective boundary conditions for the homogenization problem \eref{nonlinearhomgen}.  We zoom in at a boundary point $x_0 \in \partial \Omega$ defining,
\[ v^\e(y) = u^\e(x_0+\e y) \ \hbox{ which solves } \ 
\left\{
\begin{array}{ll}
- \grad \cdot a(y+\tfrac{x_0}{\e},\tfrac{1}{\e}Dv^\e) = 0 & \hbox{ in } \tfrac{1}{\e}(\Omega-x_0) \vspace{1.5mm}\\
v^\e(y) = g(x_0 + \e y,y+\tfrac{x_0}{\e})  & \hbox{ on } \tfrac{1}{\e}\partial(\Omega-x_0)
\end{array}\right.
\]
Now in order to have a unique equation in the limit $\e \to 0$ the following limit needs to exist,
\[ a_* (y,p) = \lim_{t \to 0} ta(y,t^{-1}p) \ \hbox{ for some } \ 1 \leq k < \infty.\]
 Note that, if said limit exists, it is always $1$-homogeneous in $p$,
\[ a_*(y,\lambda p) = \lim_{t \to 0} ta(y,(\lambda^{-1}t)^{-1} p) = \lambda a_*(y,p)\]
 In other words we need $a$ to be $1$-homogeneous in $p$ at $\infty$, then the operator $a_*$ is this limiting homogeneous profile of $a$ at $x_0$.   
 
 The above discussion motivates our assumption on the operators we study in the half-space problem.

\begin{enumerate}[(i)]
\item Periodicity:
\begin{equation}
a(x+z,p) = a(x,p) \ \hbox{ for all } \ x \in \mathbb{R}^d, z \in \mathbb{Z}^d, p \in \real^d.
\end{equation}
\item Ellipticity: for some $\lambda>0$ and all $p,q \in \mathbb{R}^{d}$
\begin{equation}
 (a(x,p) - a(x,q))\cdot(p-q) \geq \lambda|p-q|^2 \ \hbox{ and } \ |a(x,p - a(x,q)| \leq |p-q|.
 \end{equation}
 \item Positive Homogeneity: for all $x,p$ and $t>0$,
 \begin{equation}
a(x,tp) = ta(x,p)
  \end{equation}
\end{enumerate}
For convenience will also assume $a(x,p)$ is $C^1$ in $x$ so that, by the De Giorgi regularity theorem, solutions are locally $C^{1,\alpha}$ for some universal $\alpha >0$.
 
 \subsection{Regularity estimates for nonlinear equations}
 In this section we explain the regularity estimates which we use to obtain $(1)$ existence of boundary layer limits and $(2)$ the characterization of limits at rational directions.  For both results we need the De~Giorgi estimates respectively for the interior and boundary.  As is the usual approach for regularity of nonlinear equations, we can reduce to considering actually the regularity of linear equations but with only bounded measurable coefficients.  
 
 For what follows we will take $A : \real^d \to M_{d \times d}$ to be measurable and elliptic,
 \[ \lambda \leq A(x) \leq 1.\]
 Recall that results for bounded measurable coefficients imply results for solutions of nonlinear uniformly elliptic equations and for the difference of two solutions. If $u_1,u_2 \in H^{1}_{\textup{loc}}(\Omega)$ solve
 \[ - \grad \cdot a(x,\grad u_j) = 0 \ \hbox{ in } \Omega\]
 then $w = u_1 - u_2$ solves
 \begin{equation}\label{e.diffeqn}
  -\grad \cdot (A(x) \grad w) = 0 \ \hbox{ in } \Omega \ \hbox{ with } \ A(x) = \int_0^1 D_p a(x,s\grad u_1 + (1-s)\grad u_2) \ ds,
   \end{equation}
 and one can easily check that $\lambda\leq A(x) \leq 1$.  
 
 We remind that, despite the overlap of notation, the results in this section apply to solutions of scalar equations not systems.
 \begin{thm}[De Giorgi]\label{thm: degiorgi}
There is an $\alpha \in (0,1)$ and $C>0$ depending on $d,\lambda$ so that if $u$ solves,
\[ - \grad \cdot (A(x) \grad u) =0 \ \hbox{ in } \ B_1\]
then,
\[ [u]_{C^\alpha(B_{1/2})} \leq C\osc_{B_1} u\]
 \end{thm}
 A similar result holds up to the boundary for regular domains.  We say that $\Omega$ is a regular domain of $\real^d$ if there are $r_0,\mu>0$ so that for every $x \in \partial \Omega$ and every $0<r < r_0$,
 \[ |\Omega^C \cap B_{r}(x)| \geq \mu |B_r|.\]
  \begin{lem}\label{lem: bdry cont nonlinear}
Suppose that $\Omega$ is a regular domain, $r_0 \geq 1$ and $0 \in \partial \Omega$, and $\varphi \in C^{\beta}$.  There is an $\alpha_0(d,\lambda,\mu) \in (0,1)$ such that for $0<\alpha < \min\{\alpha_0,\beta\}$ there is $C(d,\lambda,\mu,\alpha)>0$ so that if $u$ solves,
\[ - \grad \cdot (A(x) \grad u) =0 \ \hbox{ in } \ B_1 \cap \Omega , \ \hbox{ with } \ u = \varphi \ \hbox{ on } \ \partial \Omega\]
then for every $r \leq 1$,
\[ \osc_{B_r} u \leq C([ \varphi]_{C^{\beta}(B_1)}+\osc_{B_1} u)r^\alpha\]
 \end{lem}
 The proof is postponed to Appendix~\ref{sec: A}.  We make a remark on the optimality of this estimate.  Using these results one can show local $C^{1,\alpha}$ estimates for solutions of non-linear uniformly elliptic equations.  Large scale $C^{1,\alpha}$ estimates are not possible due to the $x$-dependence, but in the spirit of Avellaneda-Lin~\cite{Avellaneda:1991aa} one can likely prove large scale Lipschitz estimates.  See Armstrong-Smart~\cite{Armstrong:2016aa} for the (more difficult) stochastic case, we are not aware of a citation for the periodic case.  These estimates however are for \emph{solutions}, we seem to require the result of Lemma~\ref{lem: bdry cont nonlinear} for \emph{differences} of solutions (i.e. basically it is a $C^\alpha$ estimate of a derivative).  It is not clear, therefore, whether we can do better than Lemma~\ref{lem: bdry cont nonlinear}.

  \subsection{Half-space problem}  We consider the basic well-posedness results for nonlinear problems set in half-spaces.  Consider
    \begin{equation}\label{hsnonlinear}
  \left\{
 \begin{array}{ll}
 - \grad \cdot a(x,Du) = 0 & \hbox{ in } P_n \vspace{1.5mm}\\
u = \varphi(x) & \hbox{ on } \partial P_n.
 \end{array}
 \right.
 \end{equation}
 Then the maximum principle holds. 
 \begin{lem}\label{lem: hs comparison nonlinear}
Suppose $u_1$ and $u_2$ are respectively bounded subsolutions and supersolutions of \eqref{hsnonlinear} with boundary data $\varphi_1 \leq \varphi_2$ on $\partial P_n$, then,
\[ u_1 \leq u_2 \ \hbox{ in } \ P_n.\]
 \end{lem}
\noindent The result follows from the maximum principle in bounded domains and Lemma~\ref{lem: bdry cont nonlinear}.

\section{Boundary layers limits}\label{sec: boundary layers}
In this section we will discuss the boundary layer problem for divergence form elliptic problems in rational and irrational half-spaces.  The results that we need for this paper are valid for both nonlinear scalar equations and linear systems and the proofs have only minor differences.  For that reason, in this section and the next, we will discuss both types of equations in a unified way.  We use the nonlinear notation for the PDE.  We consider the cell problem,
 \begin{equation}\label{e.cellbll}
  \left\{
 \begin{array}{ll}
 - \grad \cdot a(y,\grad v^s_n) = 0 & \hbox{ in } P_n^s \vspace{1.5mm}\\
 v^s_n = \varphi(y) & \hbox{ on } \partial P_n^s.
 \end{array}
 \right.
 \end{equation}
  We will first consider the case when $n \in S^{d-1} \setminus \real \integer^d$ is irrational.
  \subsection{Irrational half-spaces}  For linear systems the problem \eref{cellbll} in irrational half-spaces has been much studied \cite{Gerard-Varet:2012aa,Gerard-Varet:2011aa,Aleksanyan:2015aa,Aleksanyan:2017aa,Armstrong:2017aa,Prange:2013aa,Shen:2017aa}.  Typically the focus has been on the Diophantine irrational directions.  We do not give the definition, since it is not needed for our work, but basically the Diophantine condition is a quantification of the irrationality.  Under this assumption strong quantitative results can be derived for the convergence to the boundary layer limit.  
  
  For the purposes of this paper we are only interested in the qualitative result, the existence of a boundary layer limit for \eref{cellbll} in a generic irrational half-space (no Diophantine assumption).  The existence of a boundary layer tail in general irrational half-spaces was originally proven by Prange \cite{Prange:2013aa} for divergence form linear systems, and for nonlinear non-divergence form equations by the first author in \cite{Feldman:2015aa} (following the work of Choi-Kim \cite{Choi:2014aa} on the Neumann problem).  To our knowledge the case of nonlinear divergence form equations has not been studied yet.  
  
  What we would like to explain here is that the proof of \cite{Feldman:2015aa} applies also to the problems we consider in this paper, careful inspection shows that the proof of \cite{Feldman:2015aa} only required the interior regularity, continuity up to the boundary (small scale), and the $L^\infty$ estimate (or maximum principle) w.r.t. the boundary data.
  
  \begin{thm}
  Suppose that $n \in S^{d-1} \setminus \real \integer^d$.  Then there exists $\varphi_*(n)$ such that,
  \[ \sup_{s} \sup_{y \in \partial P_n} |v^s_n(y+Rn) - \varphi_*(n)| \to 0 \ \hbox{ as } \ R \to \infty.\]
  \end{thm}
    One consequence of this theorem is that, for irrational directions, we can just study $v_n = v^0_n$.  We give a sketch of the proof following \cite{Feldman:2014aa}.
  \begin{proof}(Sketch) The boundary data, and hence the solution $v^s_n$ as well, satisfies an almost periodicity property in the directions parallel to $\partial P_n$.  Given $N \geq 1$ there is a modulus $\omega_n(N) \to 0$ as $N \to \infty$ (uses $n$ irrational) so that for any $y \in \partial P_n$ there is a lattice vector $z \in \integer^d$ with $|z-y| \leq N$ and $|z \cdot n - s| \leq \omega(N)$, see \cite{Feldman:2015aa} Lemma 2.3.  
  
  Since $v_n^0(\cdot + z)$ solves the same equation in $P_n^{z \cdot n}$ we can use the up to the boundary H\"{o}lder continuity and the $L^\infty$ estimate (or maximum principle) to see that 
  \[ \|v^s_n(\cdot) - v^0_n(\cdot + z)\|_{L^\infty(P_n^s \cap P_n^{z \cdot n})} \lesssim \| \grad \varphi\|_\infty \omega_n(N)^\alpha. \]
  Sending $N$ to $\infty$ we see that if $v^0_n$ has a boundary layer limit then so does $v^s_n$ and they have the same value.
  
  Then we just need to argue for $v^0_n$. Given $y \in \partial P_n$ the same argument as above shows there is $\bar{z} \in \partial P_n$ with $|\bar{z}-y| \leq N$ and
  \[ |v^0_n(\cdot) - v^0_n(\cdot + \bar{z})| \lesssim \| \grad \varphi\|_\infty \omega_n(N)^\alpha.\] 
  Then using the $L^\infty$ estimate Lemma~\ref{lem: system max} (or the maximum principle) and the large scale interior regularity estimates, Theorem~\ref{thm: degiorgi} above for the nonlinear case or Lemma 9 in \cite{Avellaneda:1987aa} for the linear systems case,
  \begin{align*}
    \osc_{y \cdot n \geq R} v^s_n(y) &\lesssim \osc_{y \cdot n = R} v^s_n(y) \\
    &\leq \osc_{y \in B_N(0) \cap \partial P_n} v^s_n(y+Rn) + C\| \grad \varphi\|_\infty \omega_n(N)^\alpha \\
    &\lesssim  \| \grad \varphi\|_{\infty}( (N/R)^\alpha + \omega_n(N)^\alpha).
    \end{align*}
  Choosing $N$ large first to make $\omega_n(N)$ small and then $R \gg N$ gets the existence of a boundary layer limit.  
  \end{proof}

  \subsection{Rational half-spaces} Next we consider the case of a rational half-space.  Let $\xi \in \Z^d \setminus \{0\}$ be an irreducible lattice direction, and $v^s_\xi$ be the corresponding half-space problem solution.  In this case $\varphi$ is periodic with respect to a $d-1$-dimensional lattice parallel to $\partial P_\xi$.  There exist $\ell_1,\dots,\ell_{d-1}$ with $\ell_j \perp \xi$ and $|\ell_j| \leq |\xi|$ which are periods of $\varphi$.  Then by uniqueness $\ell_j$ are also periods of $v^s_\xi$.  In this special situation it is possible to show that there is a boundary layer limit with an exponential rate of convergence.
  
  We give a general set-up.  We consider the half-space problem,
   \begin{equation}
  \left\{
 \begin{array}{ll}
 - \grad \cdot a(x, \grad v) = \grad \cdot f & \hbox{ in } \real^d_+ \vspace{1.5mm}\\
 v = \psi(x') & \hbox{ on } \partial \real^d_+.
 \end{array}
 \right.
 \end{equation}
 where $\psi : \partial \real^d_+ \to \real$ and $f$ are smooth, and $\psi$, $f$, and $a(\cdot,p)$ all share $d-1$ linearly independent periods $\ell_1,\dots, \ell_{d-1} \in \partial \real^d_+$ such that,
 \[ \max_{1 \leq j \leq d-1} |\ell_j| \leq M.\]
 The operators $a$, as always, will also satisfy the assumptions of either Section~\ref{sec: linear background} or Section~\ref{sec: nonlinear background}. For now we will take $f=0$, this covers most of the situations we will run into in this paper.  Then $v$ has a boundary layer limit with exponential rate of convergence.
  \begin{lem}\label{lem: exp tail nonlinear}
  There exists a value $c_*(\psi)$ such that, 
  \[ \sup_{y \in \partial P_n} |v(y+Re_d) - c_*| \leq C(\osc \psi) e^{-cR/M},\]
  with $C,c>0$ depending only on $\lambda ,d$.
  \end{lem}
The proof of this result is the same as the proof of the analogous result, Lemma 3.1, in \cite{Feldman:2015aa}, so we only include a sketch.  The only tools necessary are the maximum principle (or $L^\infty$ estimate Lemma~\ref{lem: system max}) and the large scale interior H\"{o}lder estimates via De Giorgi-Nash-Moser for nonlinear equations (Theorem~\ref{thm: degiorgi} here) or Avellaneda-Lin for linear systems (Lemma 9 in \cite{Avellaneda:1987aa}).
  
  \begin{proof}(Sketch) Let $L \geq 1$ to be chosen, call $Q$ to be the unit periodicity cell of $\psi$ which has diameter at most $\sim M$. Apply the De Giorgi interior H\"{o}lder estimates or the Avellaneda-Lin large scale H\"{o}lder estimates to find,
  \[ \osc_{ \partial P_n + LMn} v= \osc_{y \in Q} v( y + LMn) \leq C L^{-\alpha}\osc_{P_n} u \leq C L^{-\alpha} \osc \psi \leq \frac{1}{2} \osc \psi .\]
  The second inequality is by the maximum principle or the $L^\infty$ estimate Lemma~\ref{lem: system max}, for the third inequality we have chosen $L \geq 1$ universal to make $CL^{-\alpha}  \leq 1/2$.  Then iterate the argument with the new boundary data on $\partial P_n + LMn$ with oscillation decayed by a factor of $1/2$.
  \end{proof}
  
%
We will also need a slight variant of the above result when the operator $a$ does not share the same periodicity as the boundary data, but instead has oscillations at a much smaller scale.  We assume that $\psi$ has periods $\ell_1,\dots,\ell_{d-1}$ as before, and now we also assume that there are $e_1,\dots, e_d$ which are periods of $a$ and,
\[ \max_{1 \leq j \leq d } |e_j| \leq \e.\]
For example this is the case with $a(\frac{x}{\e},p)$ when $a(\cdot,p)$ is $\integer^d$-periodic.  In this situation we do not quite have a boundary layer limit with exponential rate, but at least there is an exponential decay of the oscillation down to a scale $\sim \e^\alpha$.
  \begin{lem}\label{lem: exp tail nonlinear variant}
  There exists a value $c_*(\psi)$ such that, for some universal $\alpha \in (0,1)$ (nonlinear case) or for every $\alpha \in (0,1)$ (linear case),
  \[ \sup_{y \in \partial \real^d_+} |v(y+Re_d) - c_*| \leq C(\osc \psi) e^{-cR/M}+C\|\grad \psi\|_\infty \e^\alpha ,\]
  with $c,C>0$ universal and $C$ depending on $\alpha$ as well.
  \end{lem}
  \noindent Again the proof of this result mirrors the proof of Lemma~3.2 in \cite{Feldman:2015aa} and we do not include it.  Briefly, the idea is the same as Lemma~\ref{lem: exp tail nonlinear} except that the lattice vectors generated by $\ell_1,\dots,\ell_{d-1}$ are no longer periods of $v$, instead for each lattice vector there is a nearby vector (distance at most $\e$) which is a period of the operator.  This vector will almost be a period of $v$, with error of $\e^\alpha$ which comes from the boundary continuity estimate Lemma~\ref{lem: bdry cont nonlinear} (nonlinear) or Lemma~\ref{lem: bdry cont linear} (linear system).
  
  Finally we discuss the boundary layer problem \eref{cellbll} with non-zero right hand side $f$.  We will restrict to the case of linear systems.  We need to put a decay assumption on $f$ to guarantee even the existence of a solution.  We will assume that there are $K,b > 0 $ so that,
  \begin{equation}\label{e.fnorm}
  \sup_{y_d \geq R} |f(y)| \leq \frac{K}{R} e^{-bR/M}.
  \end{equation}
  Such assumption arises naturally, it is exactly the decay obtained for $\grad v$ when $v$ solves \eref{cellbll} with $f = 0$.  The $1/R$ polynomial decay is important since we will care about the dependence on $M \gg 1$, the exponential does not take effect until $R \gg M$ while the $1/R$ decay begins at the unit scale.
  \begin{lem}\label{lem: bdry layer rhs}
  Suppose that $f$ satisfies the bound \eref{fnorm} and $v$ is the solution of the half space problem \eref{cellbll} for a linear system satisfying the standard assumptions of Section~\ref{sec: linear background}.  Then there exists $c_*(\psi,f)$ such that,
  \[ \sup_{y \in \partial \real^d_+} |v(y+Re_d) - c_*| \leq C ((\osc \psi) +K \log M)e^{-b_0R/M}  \]
  the constants $C$ and $b_0$ depend on universal parameters as well as $b$ from \eref{fnorm}.
  \end{lem}
See the appendix and Lemma A.4 of \cite{Feldman:2015aa} for more details.

  \subsection{Interior homogenization of a boundary layer problem} In this section we will consider the \emph{interior} homogenization of half-space problems with periodic boundary data, as explained in Section~\ref{sec: outline} such a problem arises in the course of computing the directional limits of $\varphi_*$ at a rational direction.  
   \begin{equation}\label{e.half-space hom}
  \left\{
 \begin{array}{ll}
 - \grad \cdot a(\tfrac{x}{\e},\grad u^\e) = 0 & \hbox{ in } P_n \vspace{1.5mm}\\
 u^\e = \psi(x) & \hbox{ on }  \partial P_n
 \end{array}
 \right. 
 \end{equation}
 homogenizing to
 \begin{equation}\label{e.half-space hom2}
 \left\{
 \begin{array}{ll}
 - \grad \cdot a^0(\grad u^0) = 0 & \hbox{ in } P_n \vspace{1.5mm}\\
 u^0 = \psi(x) & \hbox{ on }  \partial P_n.
 \end{array}
 \right.
 \end{equation}
 Here $\psi : \partial P_n \to \real^N$, as in the previous section, will be smooth and periodic with respect to $d-1$ linearly independent translations parallel to $\partial P_n$ which we call $\ell_1,\dots,\ell_{d-1} \in \partial P_n$.  As before we call $M = \max_{j} |\ell_j|$, expecting homogenization we assume $M \gg \e$, and for convenience we can assume that $M = 1$, general results can be derived by scaling.  
 
 This problem is quite similar to the standard homogenization problem for Dirichlet boundary data, the unboundedness of the domain is compensated by the periodicity of the boundary data and by the existence of a boundary layer limit which is a kind of (free) boundary condition at infinity.  The main result of this section is the \emph{uniform} convergence of $u^\e$ to $u^0$, and hence also (importantly for us) the convergence of the boundary layer limits,

 \begin{prop}\label{prop: half-space hom}
 Homogenization holds for \eref{half-space hom} with estimates:
 \begin{enumerate}[$(i)$]
 \item \textup{(nonlinear equations)} For every $\beta \in (0,1)$, $\e \leq 1/2$, there exists $\alpha(\beta,\lambda,d)$ such that,
 \[ \sup_{P_n} |u^\e - u^0| \leq C [ \psi]_{C^\beta} \e^\alpha.\]
 \item \textup{(linear systems)} For every $\e \leq 1/2$,
 \[ \sup_{P_n} |u^\e - u^0| \leq C [ \psi]_{C^4} \e (\log \tfrac{1}{\e})^3.\]
 \end{enumerate}
 \end{prop}
\noindent We will follow the idea of \cite{Feldman:2015aa} Lemma~4.5, there is a slight additional difficulty since for divergence form nonlinear problems it is not possible to add a linear function $n \cdot x$ and preserve the solution property, even for the homogenized problem.  The $C^4$ norm we require for $\psi$ in the linear systems case is more than necessary.

The proof will use known results about homogenization of Dirichlet boundary value problems in bounded domains, specifically we consider the problem in a strip type domain,
 \begin{equation}
  \left\{
 \begin{array}{ll}
 - \grad \cdot a(\tfrac{x}{\e},\grad u^\e_R) = 0 & \hbox{ in } \Pi_n(0,R) = \{ 0 < x \cdot n < R\} \vspace{1.5mm}\\
 u^\e_R = \psi(x) & \hbox{ on } \partial \Pi_n(0,R)=  \{ x \cdot n \in \{0,R\} \},
 \end{array}
 \right.
 \end{equation}
 where we make some choice to extend $\psi$ to $x \cdot n = R$.  The solution of the homogenized problem $u^0_R$ is defined analogously.

For linear systems we have the following rate for convergence, for $R \geq 1$,
\begin{equation}\label{e.hom est lin}
 \sup_{\Pi_n(0,R)} |u^\e_R - u^0_R| \leq CR^4\| \psi\|_{C^{4}}(R^{-1}\e),
 \end{equation}
which can be derived from the rate of convergence proved in Avellaneda-Lin \cite{Avellaneda:1991aa} by scaling.  The $C^4$ regularity on $\psi$ is sufficient, we did not state the precise regularity requirement on $\psi$ which can be found in \cite{Avellaneda:1991aa}. With less regularity on $\psi$ one can also obtain an algebraic rate of convergence $O(\e^\alpha)$.

For nonlinear equations there is an algebraic rate of convergence, for any $ \beta \in (0,1)$
\begin{equation}\label{e.hom est nonlin}
 \sup_{\Pi_n(0,R)}| u^\e_R - u^0_R| \leq CR^\beta\|\psi\|_{C^{0,\beta}} (R^{-1} \e)^\alpha 
 \end{equation}
with some $\alpha = \alpha(\beta) \in (0,1)$ universal.

\begin{proof}[Proof of Proposition~\ref{prop: half-space hom}]
We define the boundary layer limits of, respectively, the $\e$-problem and the homogenized problem in \eref{half-space hom}.   We have not proven that the $\e$-problem has a boundary layer limit, however Lemma~\ref{lem: exp tail nonlinear variant} gives that the limit values are concentrated in a set of diameter $o_\e(1)$.  So we define,
 \[ \mu^\e  \in \lim_{R \to \infty} u^\e(Rn) \ \hbox{ and } \ \mu^0 = \lim_{R \to \infty} u^0(Rn),\]
where $\mu^\e$ can be any subsequential limit and satisfies, again via Lemma~\ref{lem: exp tail nonlinear variant},
\begin{equation}\label{e.nonlinear mu est}
|\mu^\e - u^\e(Rn)| \leq C\|\grad \psi\|_\infty (\e^\alpha +  e^{-cR}) \ \hbox{ (nonlinear case) } 
\end{equation}
and
\begin{equation}\label{e.linear mu est}
 |\mu^\e - u^\e(Rn)| \leq C\|\grad \psi\|_{C^{0,\nu}}( \e + C e^{-cR}) \ \hbox{ (linear system case).} 
 \end{equation}
  Instead of arguing directly with $u^\e$ and $u^0$ we consider,
 \begin{equation}\label{e.R probs}
 \left\{
\begin{array}{ll}
 -\nabla  \cdot a(\tfrac{x}{\e}, \nabla  u^{\e}_{R})=0 & \hbox{ in } \Pi_n(0,R) ,\vspace{1.5mm}\\
 u_{R}^{\e} =\psi(x) &  \hbox{ on } x\cdot n = 0, \vspace{1.5mm}\\
 u_R^{\e} = \mu^\e &   \hbox{ on } x \cdot n = R.
\end{array}
\right. 
\end{equation}
and, for $j \in \{0,\e\}$
\begin{equation}
\left\{
\begin{array}{ll}
 -\nabla  \cdot a^0( \nabla  u^{0}_{R,j})=0 & \hbox{ in } \Pi_n(0,R) ,\vspace{1.5mm}\\
 u_{R,j}^{0} =\psi(x) &  \hbox{ on } x\cdot n = 0, \vspace{1.5mm}\\
 u_{R,j}^{0} = \mu^j &   \hbox{ on } x \cdot n = R.
\end{array}
\right. 
\end{equation}
We will choose $R = R(\e)$ below to balance the various errors.  The error in replacing $u^\e$ by $u^\e_R$,
\[ |u^\e(x) - u^\e_R(x)| \leq C\|\grad \psi\|_\infty (\e^\alpha +  e^{-cR}) \ \hbox{ for } \ x \in \Pi_n(0,R),\]
and replacing $u^0$ by $u^0_{R,0}$,
\[ |u^0(x) - u^0_{R,0}(x)| \leq C(\osc \psi)  e^{-cR} \ \hbox{ for } \ x \in \Pi_n(0,R), \]
the estimates holds on $\partial \Pi_n(0,R)$ by \eref{nonlinear mu est} (or for linear we use \eref{linear mu est} instead), and therefore by the maximum principle (or by Lemma~\ref{lem: system max} for linear systems) they hold on the interior as well.  To estimate the error in replacing $u^0_{R,0}$ by $u^0_{R,\e}$ we need to estimate the difference $\mu^\e - \mu^0$, which is basically the goal of the proof, this will be achieved below.

By Lemma~\ref{lem: bdry cont nonlinear} (or Lemma~\ref{lem: flat reg} in the linear systems case) there exists a universal $\delta_0(\lambda,d)>0$ so that if $B$ is uniformly elliptic and $q$ solves,
\begin{equation}\label{e.q unif}
\left\{
\begin{array}{ll}
 -\nabla  \cdot (B(x) \grad q)=0 & \hbox{ in } \Pi_n(0,1) ,\vspace{1.5mm}\\
 q =0 &  \hbox{ on } x\cdot n = 0, \vspace{1.5mm}\\
 |q| =1  &   \hbox{ on } x \cdot n = 1,
\end{array}
\right. \ \hbox{ then } \ |q(x)| \leq \frac{1}{2} \ \hbox{ for } \ x \cdot n \leq \delta_0.
\end{equation}
Now call, 
\[ q^\e = u^0_{R,0} - u^{0}_{R,\e} \ \hbox{ which solves } \ \left\{
\begin{array}{ll}
 -\nabla  \cdot (B(x) \grad q^\e)=0 & \hbox{ in } 0<x\cdot n<R ,\vspace{1.5mm}\\
 q^\e =0 &  \hbox{ on } x\cdot n = 0, \vspace{1.5mm}\\
 q^\e =  \mu^0 - \mu^\e &   \hbox{ on } x \cdot n= R
\end{array}
\right.
\]
with $B(x) = A^0$ in the linear case or,
\[ B(x) =  \int_0^1 Da^0( t \grad u^0_{R,0}(x) + (1-t) \grad u^0_{R,\e}(x)) \ dt  \ \hbox{ uniformly elliptic,}\]
in the nonlinear case.  Now $\frac{1}{|\mu^0 - \mu^\e|}q(Rx)$ solves an equation of the type \eref{q unif} and so,
\[ |q(\delta_0 Rn)| \leq \tfrac{1}{2}|\mu^0  - \mu^\e|.\]
Now we apply the homogenization error estimates \eref{hom est nonlin} and \eref{hom est lin} for the domain $\Pi_n(0,R)$ to \eref{R probs}
\[ |u^{0}_{R,\e} - u^\e_R| \leq CR \|\grad \psi\|_{\infty}(R^{-1}\e)^\gamma\]
or respectively in the linear system case,
\[ |u^{0}_{R,\e} - u^\e_R| \leq CR^4 \|\psi\|_{C^4}(R^{-1}\e). \]
Now we estimate the error in $\mu^\e - \mu^0$, for the nonlinear case,
\begin{align*}
 | \mu^\e - \mu^0 | &\leq |u^\e(\delta_0 Rn) - u^0(\delta_0 Rn)| + C\|\grad \psi\|_\infty(\e^\alpha + e^{-cR}) \\
 &\leq |u^\e_R(\delta_0 Rn) - u^{0}_{R,\e}(\delta_0 Rn)| + |q^\e(\delta_0Rn)|+C\|\grad \psi\|_\infty(\e^\alpha + e^{-cR}) \\
 &\leq CR \|\grad \psi\|_{\infty}(R^{-1}\e)^\gamma + \tfrac{1}{2} |\mu^\e - \mu^0|+ C\|\grad \psi\|_\infty(\e^\alpha + e^{-cR}).
 \end{align*}
 Moving the middle term above to the left hand side we find, 
 \[ | \mu^\e - \mu^0 |  \leq C \|\grad \psi\|_\infty( R(R^{-1}\e)^\gamma + \e^\alpha + e^{-cR}) \leq C  \|\grad \psi\|_\infty \e^{\alpha'} \]
 where finally we have chosen $R = C \log \frac{1}{\e}$ and $\alpha' < \min\{ \alpha, \gamma \}$.  The same argument in the linear case yields,
 \[ | \mu^\e - \mu^0 | \leq C [\psi]_{C^{4}}( R^4(R^{-1}\e) + \e + e^{-cR}) \leq C[\psi]_{C^{4}} \e (\log \tfrac{1}{\e})^3.\]

\end{proof}

\section{Asymptotics near a rational direction}\label{sec: asymptotics}
We study asymptotic behaviour of the cell problems as $n \in S^{d-1}$ approaches a rational direction $\xi\in \mathbb{Z}^d\backslash\{0\}$. We recall $v^s_{\xi}$ the solution of the cell problem,
\begin{equation}\label{cell1}
\left\{\begin{aligned}
& -\nabla  \cdot a(x+s\xi, \nabla v^s_{\xi}) =0 &\text{ in }& P_\xi ,\\
&  v^s_{\xi}(x)=\varphi(x+s\xi) &\text{ on }& \partial P_\xi .
\end{aligned}
\right.\end{equation}
The boundary layer limit of the above cell problem depends on the parameter $s$ and we define,
\begin{equation}\label{e.phi*}
\varphi_{*}(\xi,s):=\lim_{R\rightarrow \infty}v^s_{\xi}(x+R\xi)
\end{equation}
which is well-defined and the limit is independent of $x$, see Lemma~\ref{lem: exp tail nonlinear}.  It follows from Bezout's identity that $\varphi_*$ is a $1/|\xi|$ periodic function on $\real$, see Lemma 2.9 in \cite{Feldman:2015aa}.  As long as we can we will combine the arguments for linear systems and nonlinear equations.

\subsection{Regularity of \texorpdfstring{$\varphi_*(\xi,\cdot)$}{Lg}}
To begin we need to establish some regularity of $\varphi_*(\xi,\cdot)$. For quantitative purposes it is important to control the dependence of the regularity on $|\xi|$.  We just state the results, postponing the proofs till the end of the section.  A modulus of continuity for $\varphi_*(\xi,\cdot)$ which is uniform in $|\xi|$ is not difficult to establish.  This follows from the continuity up to the boundary Lemma~\ref{lem: bdry cont nonlinear} (or Lemma~\ref{lem: flat reg}) and the maximum principle Lemma~\ref{lem: hs comparison nonlinear} (or the $L^\infty$ estimate Lemma~\ref{lem: system max}). 
\begin{lem}\label{lem: reg holder}
The boundary layer limits $\varphi_*(\xi,s)$ are continuous in $s$.
\begin{enumerate}[(i)]
\item \textup{(Nonlinear equations)}
\[ [ \varphi_*(\xi,\cdot)]_{C^\alpha} \leq C \| \grad \varphi\|_\infty ,\]
which holds for some universal $C \geq 1$ and $\alpha \in (0,1)$.  
\item \textup{(Linear systems)}  H\"{o}lder estimates as above holds for all $\alpha \in (0,1)$ and moreover,
\[ \| \frac{d}{ds}\varphi_*(\xi,\cdot)\|_\infty \leq C \| \grad \varphi\|_{C^{0,\nu}} \ \hbox{ for any } \ 0 < \nu \leq 1. \]
\end{enumerate}
\end{lem}
To optimize our estimates, in the linear case, we will also need higher regularity of $\varphi_*$ which is (almost) uniform in $|\xi|$, this is somewhat harder to establish.
\begin{lem}\label{lem: reg smooth}
\textup{(Linear systems)} For any $\xi \in \integer^d \setminus \{0\}$, suppose $\varphi_{*}(\xi,s)$ is defined as above. Then for all $j\in\mathbb{N}^d$ and any $\nu >0$ there exists some constant $C_j$ universal such that,
\[\sup_s |\frac{d^j}{d s^j} \varphi_{*}(\xi,s)|\leq    C_j\|\varphi\|_{C^{j,\nu}}\log^{j} (1+|\xi|).\]
\end{lem}
Note that Lemma~\ref{lem: reg smooth} is a bit weaker than Lemma~\ref{lem: reg holder} in the case $j=1$, this is because we take a different approach which is suboptimal in the $j=1$ case, it is not clear if the logarithmic terms are necessary when $j >1$.  The proof is similar to Lemma 7.2 in \cite{Feldman:2015aa}, taking the derivative of $v^s_\xi$ with respect to $s$ and estimating based on the PDE.  Probably more precise Sobolev estimates are possible but we did not pursue this.

\subsection{Intermediate scale asymptotics}\label{sec3}
Consider an irrational direction $n$ close to a lattice direction $\xi \in \integer^d \setminus \{0\}$.  Let $\e>0$ small and we write,
\[ n = (\cos\e)\hat\xi - (\sin\e)\eta \ \hbox{ for some } \ \xi \in \integer^d \setminus \{0\} \ \hbox{ and a unit vector } \ \eta \perp \xi.\]
 We will assume below that $|\e| \leq \pi/6$. We consider the cell problem in $P_n$
\begin{equation}\label{e.vn}
\left\{\begin{aligned}
& -\nabla \cdot a(y,\grad v_n) =0 &\text{ in }& P_{n}  ,\\
&  v_n=\varphi( {y}) &\text{ on }&  \partial P_{n}.
\end{aligned}
\right.\end{equation}
The first step of the argument is to show, with error estimate, that the boundary layer limit of $v_n$ is close to the boundary layer limit of the problem
\begin{equation}\label{e.vI}
\left\{\begin{aligned}
& -\nabla \cdot a(\tfrac{y}{\tan \e},\grad v^I_n) =0 &\text{ in }& P_{n}  ,\\
&  v_n^I=\varphi_*(\xi, y \cdot \eta) &\text{ on }&  \partial P_{n}.
\end{aligned}
\right.\end{equation}
The solution $v_n^I$ approximates $v_n$, asymptotically as $\e \to 0$, starting at an intermediate scale $1 \ll R \ll 1/\e$ away from $\partial P_n$.  The argument is by direct comparison of $v_n$ with $v_\xi^s$ in their common domain.  

Since the problem \eref{vI} has a boundary layer of size uniform in $\e$ we can replace, again with small error, by a problem in a fixed domain
\begin{equation}\label{e.we}
\left\{\begin{aligned}
& -\nabla \cdot a(\tfrac{y}{\tan\e},\grad w_{\xi,\eta}^\e) =0 &\text{ in }& P_{\xi}  ,\\
&  w_{\xi,\eta}^\e=\varphi_*(\xi, y \cdot \eta) &\text{ on }&  \partial P_{\xi}.
\end{aligned}
\right.\end{equation}
We remark that for both \eref{vI} and \eref{we} we have not proven the existence of a boundary layer limit, rather we use Lemma~\ref{lem: exp tail nonlinear variant}.  For convenience we will state estimates on $\lim_{R\to\infty} v_n^I(Rn)$ or on $\lim_{R \to \infty} w_{\xi,\eta}(R\hat\xi)$, but technically we will mean that the estimate holds for every sub-sequential limit.

\begin{prop}\label{prop: int scale}
Let $\xi \in \integer^d \setminus \{0\}$ and $n = (\cos\e)\hat\xi - (\sin\e)\eta$ with $\e >0$ small and a unit vector $\eta \perp \xi$.
\begin{enumerate}[$(i)$]
\item \textup{(Nonlinear equations)} There is universal $\alpha \in (0,1)$ such that
\[ |\varphi_*(n) - \lim_{R \to \infty} w_{\xi,\eta}^\e(R\hat\xi)| \lesssim \|\grad \varphi\|_\infty |\xi|^\alpha \e ^\alpha,\]
where we mean that the estimate holds for any sub-sequential limit of $w_{\xi,\eta}^\e(R\hat\xi)$ as $R \to \infty$.  
\item \textup{(Linear systems)} For every $\alpha \in (0,1)$ and any $\nu>0$
\[ |\varphi_*(n) - \lim_{R \to \infty} w_{\xi,\eta}^\e(R\hat\xi)| \lesssim_{\alpha,\nu} [\varphi]_{C^{1,\nu}} |\xi|^\alpha \e ^\alpha,\]
where again we mean that the estimate holds for any sub-sequential limit of $w_{\xi,\eta}^\e(R\hat\xi)$ as $R \to \infty$.  
\end{enumerate}
\end{prop}

The first step is to compare the boundary layer limits of \eref{vn} and \eref{vI}.

\begin{lem}\label{lem: vn vI}
Fix any $x\in \partial P_{n }$, $1 \leq R \leq 1/\e$ and let $s=x\cdot\eta\tan \e$. 
\begin{enumerate}[$(i)$]
\item \textup{(Nonlinear equations)} There is universal $\alpha \in (0,1)$ such that
\[|v_n-v_\xi^{s}|(x+Rn)\lesssim \| \grad \varphi\|_\infty(R\e)^\alpha.\]
\item \textup{(Linear systems)} For every $\alpha \in (0,1)$
\[ |v_n-v_\xi^{s}|(x+Rn)\lesssim_\alpha\| \grad \varphi\|_{\infty}(R\e)^\alpha. \]
\end{enumerate}
\end{lem}
Before we go to the proof let us derive some consequences of the Lemma. Let's assume that $\|\grad \varphi\|_\infty \leq 1$ to simplify the exposition, the general inequalities can of course be derived by rescaling.  Combining Lemma~\ref{lem: exp tail nonlinear} with Lemma~\ref{lem: vn vI} we find that for any $ R \geq 1$,
\[|v_n(x+Rn)-\varphi_*(\xi, x \cdot \eta \tan \e)|\lesssim \left[(R\e)^\alpha+e^{-cR/|\xi|}\right] \ \hbox{ for } \ x \in \partial P_n.\]
Choosing $R = |\xi| \log \frac{1}{\e}$ we obtain,
\[ |v_n(x+Rn)-\varphi_*(\xi, x \cdot \eta \tan \e)|\lesssim  |\xi|^\alpha \e^\alpha \ \hbox{ for } \ x \in \partial P_n,\]
either for a slightly smaller universal $\alpha$ in the nonlinear case, or again for every $\alpha \in (0,1)$ in the case of linear systems.  

Now consider the rescaling
\[ \tilde{v}_n^I(y) = v_n(Rn + \tfrac{y}{\tan \e}) \ \hbox{ defined for } \ y \in P_n\]
which solves
\begin{equation}\label{e.vI2}
\left\{\begin{aligned}
& -\nabla \cdot a(Rn+\tfrac{y}{\tan\e},\grad \tilde{v}^I_n) =0 &\text{ in }& P_{n}  ,\\
&  |\tilde{v}_n^I - \varphi_*(\xi, y \cdot \eta)| \leq C |\xi|^\alpha \e^\alpha &\text{ on }&  \partial P_{n}.
\end{aligned}
\right.\end{equation}
This is almost the same as equation \eref{vI} solved by $v_n^I$.   First assume $Rn = 0 \bmod \integer^d$ to make things simple, then the $L^\infty$-estimate Lemma~\ref{lem: system max} (or the maximum principle) implies that
\begin{equation}\label{e.vvtildeest}
 \sup_{P_{n}}|v_n^I - \tilde{v}_n^I| \lesssim  |\xi|^\alpha \e^\alpha.
 \end{equation}
To be precise we should consider instead $\tilde{v}^I_n(Rn - [Rn] + \frac{y}{\tan \e})$ where $[Rn]$ is the representative of $Rn\bmod\integer^d$ in $[0,1)^d$. Then we would instead have,
\[ |\tilde{v}_n^I(y) - \varphi_*(\xi, y' \cdot \eta)| \lesssim |\xi|^\alpha \e^\alpha \ \text{ on } \   \partial P_{n} \]
for some $|y' - y| \leq \sqrt{d} \tan \e$.  Then applying the regularity of $\varphi_*$ from Lemma~\ref{lem: reg holder} we get the same estimate as before \eref{vvtildeest}. 

\begin{proof}[Proof of Lemma~\ref{lem: vn vI}]
Let us call the cone domains,
\[K(x):=(P_\xi+x)\cap P_{n} \ \hbox{ and } \ K_R(x)=K(x)\cap B_R(x), \]
 we may simply write $K,K_R$ if $x=0$. 
 Let $x_0 \in \partial P_n$, we compute using $n \cdot x_0 = 0$ and $n = (\cos \e )\hat\xi - (\sin \e) \eta$ that,
\[ x_0 \cdot \hat \xi = (x_0 \cdot \eta)\tan \e . \]
Let $x \in \partial K(x_0)$,  then $x \in \partial P_n$ (or $x \in \partial P_\xi + x_0$) and there exists $y \in  \partial P_\xi + x_0$ (or respectively $\partial P_n$) with,
\[ |x-y| \leq |x-x_0|\sin \e \leq \e |x-x_0|.\]
 {\it Nonlinear equations:} Applying the De Giorgi boundary continuity estimates Lemma~\ref{lem: bdry cont nonlinear} for small enough $\alpha \in (0,1)$ universal, for all $x \in \partial K(x_0)$,
\[ |v_\xi^s(x) - v_n(x)| \leq |v_\xi^s(x) - \varphi(y)|+|\varphi(y) - v_n(x)| \lesssim \|\grad \varphi\|_{\infty}\e^\alpha|x-x_0|^\alpha.\]
Now since $v_\xi^s(x) - v_n(x)$ is a difference of solutions we can apply the boundary continuity estimate from Lemma~\ref{lem: bdry cont nonlinear} again,
\[ |v_\xi^s(x) - v_n(x)| \lesssim \|\grad \varphi\|_{\infty}\e^\alpha|x-x_0|^\alpha \ \hbox{ for } \ x \in K(x_0)\]
with perhaps a slightly smaller $\alpha(d,\lambda)$.

\medskip

\noindent  {\it Linear systems:} 
We have, by almost the same argument as above now using instead Lemma~\ref{lem: flat reg}, for any $\alpha \in (0,1)$
\[ |v_\xi^s(x) - v_n(x)| \lesssim \|\grad \varphi\|_{\infty}\e^\alpha|x-x_0|^\alpha \ \hbox{ on } \ \partial K(x_0). \]
Now by the Poisson kernel bounds in $K(x_0)$, Lemma~\ref{lem: PK bounds 1} and Lemma~\ref{lem: PK bounds 2}, for a slightly smaller $\alpha$ and $\e$ sufficiently small depending on $\alpha$
\[ |v_\xi^s(x) - v_n(x)| \lesssim \|\grad \varphi\|_{\infty} \e^\alpha |x-x_0|^\alpha \ \hbox{ for } \ x \in K(x_0). \]
The remainder of the proof is the same as the case of scalar equations.

\end{proof}

To complete the proof of Proposition~\ref{prop: int scale} we just need to compare the solutions of \eref{vI} and \eref{vn}.  The width of the boundary layer is now of uniform size in $\e$ so this is not a problem, we will just need to use the boundary continuity estimate (Lemmas \ref{lem: bdry cont linear} and \ref{lem: bdry cont nonlinear}) and the continuity estimate of $\varphi_*(\xi,\cdot)$ Lemma~\ref{lem: reg holder}.
\begin{lem}\label{lem: vI we}
The following estimates hold for the boundary layers of $v_n^I$ and $w^\e_{\xi,\eta}$.
\begin{enumerate}[$(i)$]
\item \textup{(Nonlinear equations)} There is $\alpha \in (0,1)$ universal such that
\[ |\lim_{R \to \infty} v_n^I(Rn) - \lim_{R \to \infty} w_{\xi,\eta}^\e(R\hat\xi)| \lesssim \|\grad \varphi\|_\infty |\xi|^\alpha \e^\alpha ,\]
where technically we mean that the estimate holds for any pair of sub-sequential limits.
\item \textup{(Linear systems)} For every $\alpha \in (0,1)$ and any $\nu>0$
\[ |\lim_{R \to \infty} v_n^I(Rn) - \lim_{R \to \infty} w_{\xi,\eta}^\e(R\hat\xi)| \lesssim_{\alpha,\nu} [\varphi]_{C^{1,\nu}} |\xi|^\alpha \e^\alpha ,\]
where technically we mean that the estimate holds for any pair of sub-sequential limits.
\end{enumerate}
\end{lem}
\begin{proof}
We compare the two solutions in their common domain.  Calling as before $K = P_n \cap P_\xi$ and,
\[ u = v_n^I - w_{\xi,\eta}^\e.\]

\medskip

 \noindent {\it Nonlinear equations:} We have that
\[ - \grad \cdot( A(x) \grad u )= 0 \ \hbox{ in } \ K \ \hbox{ with some $\lambda \leq A(x) \leq 1$ as in \eref{diffeqn}}.\]
We compute the error on $\partial K$ in the same way that we did in Lemma~\ref{lem: vn vI}.    Using Lemma~\ref{lem: bdry cont nonlinear} we find for $x \in \partial K$,
\[ |u(x)| = |v_n^I(x) - w_{\xi,\eta}^\e(x)| \lesssim \|\varphi_*(\xi,\cdot)\|_{C^{\alpha'}}\e^{\alpha}|x|^\alpha  \lesssim \|\grad \varphi\|_\infty \e^\alpha |x|^\alpha,\]
where $\alpha'$ is the, universal, continuity modulus from Lemma~\ref{lem: reg holder} and $\alpha < \alpha'$.  Next we use the De Giorgi boundary continuity estimate, Lemma~\ref{lem: bdry cont nonlinear} to obtain, again with a slightly smaller $\alpha$,
\begin{equation}\label{e.bc we}
 |u(x)|\lesssim \|\grad \varphi\|_\infty \e^\alpha |x|^\alpha \ \hbox{ for } \ x \in K. 
 \end{equation}
Next we use that the size of the boundary layers for $v_n^I$ and $w_{\xi,\eta}^\e$ are uniformly bounded in $\e$, via Lemma~\ref{lem: exp tail nonlinear variant}, to find for all $R_0 \geq 1$,
\[ \sup_{y \in \partial P_n}|v_n^I(y+R_0n) - \lim_{R \to \infty} v_n^I(Rn)| \lesssim \|\varphi_*(\xi,\cdot)\|_{C^{\alpha'}} \e^\alpha + (\osc \varphi_*) e^{-R_0/|\xi|},\]
where again we mean that the estimate holds for any sub-sequential limit of $v_n^I(Rn)$.  An analogous estimate holds for $w_{\xi,\eta}^\e$ replacing $R n$ with $R\hat\xi$.  Using our assumption that $\e \leq  \pi /4$ we have $n \cdot \hat \xi \geq 1/\sqrt{2}$ and so we have,
\begin{align}\label{e.bdry layer n xi}
\max \{ |v_n^I(R_0\hat\xi) - \lim_{R \to \infty} v_n^I(Rn)|,&|w_{\xi,\eta}^\e(R_0\hat\xi) - \lim_{R \to \infty} w_{\xi,\eta}(R\hat\xi)|\} \lesssim  \\
&\|\varphi_*(\xi,\cdot)\|_{C^{\alpha'}} \e^\alpha + (\osc \varphi_*) e^{-R_0/|\xi|}. \notag
\end{align}
Finally we combine \eref{bc we} with \eref{bdry layer n xi}, choosing $R_0 = |\xi| \log \frac{1}{|\xi|\e}$, to find,
\begin{align*}
 |\lim_{R \to \infty} v_n^I(Rn) - \lim_{R \to \infty} w_{\xi,\eta}^\e(R\hat\xi)| &\leq |v_n^I(R_0\hat\xi) - w_{\xi,\eta}^\e(R_0\hat\xi)| + C\|\grad \varphi\|_\infty |\xi|^\alpha\e^\alpha\\
 &\lesssim \|\grad \varphi\|_\infty \e^\alpha R_0^\alpha \\
 & \lesssim \|\grad \varphi\|_\infty |\xi|^\alpha \e^\alpha (\log\tfrac{1}{|\xi|\e})^\alpha.
 \end{align*}
 Making $\alpha$ slightly smaller we can remove the logarithmic term.
 
 \medskip 
 
 \noindent {\it Linear systems:} We have that
\[ - \grad \cdot( A(x) \grad u )= 0 \ \hbox{ in } \ K.\]
Using Lemma~\ref{lem: flat reg} we find, for $x \in \partial K$ and any $\nu >0$,
\[ |u(x)| = |v_n^I(x) - w_{\xi,\eta}^\e(x)| \lesssim_\alpha \|\grad \varphi_*(\xi,\cdot)\|_{\infty}\e^\alpha |x|^\alpha \lesssim_\nu \|\grad \varphi\|_{C^{0,\nu}} \e^\alpha |x|^\alpha.\]
  By the Poisson kernel bounds in $K$, Lemma~\ref{lem: PK bounds 1} and Lemma~\ref{lem: PK bounds 2}, we have for a slightly smaller $\alpha \in (0,1)$ and $\e$ sufficiently small depending on $\alpha$
\[ |u(x)| \lesssim_\alpha [\varphi]_{C^{1,\nu}}\e^\alpha |x|^\alpha \ \hbox{ for } \ x \in K. \]
The remainder of the proof is the same as the case of scalar equations.
\end{proof}

\subsection{Interior homogenization of the intermediate scale problem}\label{subsect homo inter scale} We take $\e\rightarrow 0$ in \eref{we} and derive the second cell problem,
\begin{equation}\label{cell2e}
\left\{\begin{aligned}
& -\nabla  \cdot a( \tfrac{x}{\tan\e},\nabla w_{\xi,\eta}^\e)=0 &\text{ in }&  P_{\xi},\\
& w_{\xi,\eta}^\e(x)=\varphi_{*}(\xi,x\cdot\eta) &\text{ on }& \partial P_{\xi} 
\end{aligned}
\right.
\end{equation}
which homogenizes to
\begin{equation}\label{cell2}
\left\{\begin{aligned}
& -\nabla  \cdot a^0( \nabla w_{\xi,\eta})=0 &\text{ in }&  P_{\xi},\\
& w_{\xi,\eta}(x)=\varphi_{*}(\xi,x\cdot\eta) &\text{ on }& \partial P_{\xi} .
\end{aligned}
\right.
\end{equation}
where $a^0$ is the homogenized operator associated with $a(\tfrac{x}{\e},\cdot)$.

We make the definition
\[L(\xi,\eta)=\lim_{R\rightarrow \infty} w_{\xi,\eta}(x+R\xi).\]
As we will show below $L(\xi,\cdot)$ is the limiting $0$-homogeneous profile of $\varphi_*$ at the direction $\xi$,
\[ \lim_{n \to \hat\xi} \varphi_*(n) = L(\xi,\eta) \ \hbox{ for $n$ irrational } \ n \to \hat \xi \ \hbox{ and } \  \frac{ \hat\xi - n}{|\hat \xi - n|} \to \eta.\]
This characterization is the first main result of the paper Theorem~\ref{main1}.

We make a further remark about the second cell problem in \eref{cell2}.  It is straightforward to see that $w_{\xi,\eta}$ is actually a function only of two variables $x \cdot \xi$ and $x \cdot \eta$. The boundary data $\varphi_*(\xi,x \cdot \eta)$ is invariant with respect to translations which are perpendicular to both $\xi$ and $\eta$, and so by uniqueness the solution $w_{\xi,\eta}$ is invariant in those directions as well. Note that we are using the spatial homogeneity of the operator here, the same is not true of $w^\e_{\xi,\eta}$.  This property was useful in \cite{Feldman:2015aa} since solutions of nonlinear non-divergence form elliptic problems in dimension $d=2$ have better regularity properties.  Although we do not use this in a significant way here, we point it out anyway since it could be potentially useful in the future.

Now we state and prove the quantitative version of Theorem~\ref{main1}:
\begin{thm}\label{mainineq}
Let $\xi \in \integer^d \setminus \{0\}$ be irreducible and $n = (\cos \e ) \hat \xi - (\sin \e ) \eta$ be an irrational direction then,
\begin{enumerate}[$(i)$]
\item (Nonlinear equations) There is a universal $\alpha \in (0,1)$ such that
\[|\varphi_*(n)-L(\xi,\eta)|\lesssim \| \grad \varphi\|_{\infty}|\xi|^\alpha\e^\alpha. \]
\item (Linear systems) For every $\alpha \in (0,1)$
\[  |\varphi_*(n)-L(\xi,\eta)|\lesssim_\alpha [ \varphi]_{C^5}|\xi|^\alpha\e^\alpha \]
\end{enumerate}
\end{thm}
\begin{proof}
The ingredients have all been established elsewhere, we just need to combine them.  By Proposition~\ref{prop: half-space hom}, homogenization of problems in half-space type domains, for nonlinear equations,
\[ \sup_{P_\xi} |w_{\xi,\eta} - w^\e_{\xi,\eta}| \lesssim [\varphi _*]_{C^\beta} \e^\alpha \lesssim \| \grad \varphi\|_{\infty} \e^\alpha \ \hbox{ for some universal $\beta,\alpha \in (0,1)$,} \]
or in the linear systems case,
\[ \sup_{P_\xi} |w_{\xi,\eta} - w^\e_{\xi,\eta}| \lesssim_\alpha  [ \varphi_*]_{C^4} \e^\alpha \lesssim  \ \hbox{ for every $\alpha \in (0,1)$.}  \]
We have used Lemmas~\ref{lem: reg holder} and \ref{lem: reg smooth} to obtain the necessary regularity estimates of $\varphi_*(\xi,\cdot)$.  The factors of $\log(1+|\xi|)$ in Lemma~\ref{lem: reg smooth} can be absorbed by making $\alpha$ slightly smaller.
\end{proof}

\subsection{Proofs of regularity estimates of \texorpdfstring{$\varphi_*$}{Lg}}  We return to prove the regularity estimates of $\varphi_*$ Lemma~\ref{lem: reg holder} and Lemma~\ref{lem: reg smooth}.  The H\"{o}lder regularity Lemma~\ref{lem: reg holder} is relatively straightforward, while the higher regularity Lemma~\ref{lem: reg smooth} requires some more careful estimates.
\begin{proof}[Proof of Lemma~\ref{lem: reg holder}]
We will show an upper bound for $|\varphi_*(\xi,h) - \varphi_*(\xi,0)|$ with $h<0$, the proof works also for nonzero $s$ and $h \in \real$.   Consider $v_\xi^0$ a solution in $P_\xi$ and $v_\xi^h$ a solution in $P_\xi + h \hat \xi \supset P_\xi$.  By the boundary continuity estimates for $v_\xi^h$ for every $y \in \partial P_\xi$,
\[ |v_\xi^h(y) - v_\xi^0(y)| = |v_\xi^h(y) - \varphi(y)| \leq |v_\xi^h(y) - \varphi(y - h \hat \xi)| +  \| \grad \varphi\|_\infty h \leq C\| \grad\varphi\|_\infty h^\alpha,\] 
for some $\alpha \in (0,1)$ by Lemma~\ref{lem: bdry cont nonlinear}.  For the case of linear systems we have similarly,
\[  |v_\xi^h(y) - v_\xi^0(y)|  = |v_\xi^h(y) - \varphi(y)| \leq C [ \varphi]_{C^{1,\nu}} h \]
for any $\nu>0$ by the boundary gradient estimates for smooth coefficient linear systems.  Then the maximum principle, or respectively the $L^\infty$ estimate for systems Lemma~\ref{lem: system max}, implies the same bound holds in all of $P_\xi$ and therefore also for the boundary layer limits.
\end{proof}
\begin{proof}[Proof of Lemma~\ref{lem: reg smooth}]  In order to get estimates on higher derivatives of $v^s_\xi$ in $s$ the method for Lemma~\ref{lem: reg holder} doesn't work, we need to differentiate in the equation. Since we only consider only one normal direction $\xi \in \integer^d \setminus \{0\}$ we drop the dependence $v^s = v^s_\xi$ on $\xi$.  We denote derivatives with respect to $s$ by $\partial$ and then,
\begin{equation}\label{eqn: s deriv}
\left\{
\begin{array}{ll}
 - \grad \cdot (A(x + s \hat \xi) \grad \partial^k v^s) = \grad \cdot f & \hbox{ in } \ P_\xi \vspace{1.5mm}\\
 \partial^kv^s = ( \hat \xi \cdot \grad)^k \varphi(x + s \hat \xi) & \hbox{ on } \ \partial P_\xi,
\end{array}
\right.
\end{equation}
where $f$ involves derivatives $\partial^jv$ for $j <k$,
\[ f = \sum_{j=0}^{k-1}  { k \choose j}(\hat \xi \cdot \grad)^{k-j}A(x + s \hat \xi) \grad \partial^jv^s.\]
Let $p>d$ arbitrary but fixed.  We will suppose, inductively, that we can prove for any $R\geq 0$ and every $j <k$,
\[  \sup_{y \in \partial P_\xi, R' \geq R}\|\grad \partial^j v^s\|_{L^p_{avg}(B_{R'/2}(y+R'\hat\xi))}  \leq C_j [ \varphi ]_{C^{j+1,\nu}}\frac{1}{R}\log^j(1+|\xi|) e^{-c_jR/|\xi|} .\]
the constants depend on $j,[A]_{C^j}$ and universal parameters.  The case $R \leq 1$ corresponds basically to an $L^\infty$ bound on $P_\xi$.

Then by Lemma~\ref{lem: rhs infty}
\begin{equation}\label{e.induct Linfty}
 \|\partial^kv^s\|_{L^\infty(P_\xi)} \leq C\|( \hat \xi \cdot \grad)^k \varphi\|_{\infty} + C\log^k(1+|\xi|)[\varphi]_{C^{k,\nu}}.
 \end{equation}
Furthermore, by Lemma~\ref{lem: layer limit rhs}, $\partial^kv^s$ has a boundary layer limit $\mu_k = \frac{d^k}{ds^k}\varphi_*(\xi,s)$ with,
\[|\partial^kv^s - \mu_k| \leq C\log^k(1+|\xi|)[\varphi]_{C^{k,\nu}}e^{-cR/|\xi|}.\]
Now we aim to establish the inductive hypothesis. The following argument will also establish the base case when $j=0$.  First in the case $R \leq 1$. This follows from \eref{induct Linfty} and the up to the boundary gradient estimates (Lemma~\ref{lem: flat reg}), 
\[ \|\grad \partial^kv^s\|_{L^\infty(P_\xi)} \leq C\|( \hat \xi \cdot \grad)^k \varphi\|_{C^{1,\nu}} +C \log^k(1+|\xi|)[\varphi]_{C^{k,\nu}} \leq C\log^k(1+|\xi|)[\varphi]_{C^{k+1,\nu}} \]
In the case $R \geq 1$, by the Avellaneda-Lin large scale interior $W^{1,p}$ estimates and the inductive hypothesis,
\begin{align*}
 \|\grad \partial^kv^s\|_{L^p_{avg}(B_{R/2}(y+R\hat\xi))} &\leq  C\frac{1}{R} \osc_{B_{3R/4}(y+R\hat\xi)}\partial^kv^s + \|f\|_{L^p_{avg}(B_{3R/4}(y+R\hat\xi))} \\
 &\leq  C\frac{1}{R}\log^k(1+|\xi|)[\varphi]_{C^{k,\nu}}e^{-cR/|\xi|}.
 \end{align*}
Combining the cases $R\leq 1$ and $R \geq 1$ establishes the inductive hypothesis for $j=k$. The bound on $\|\partial^kv^s\|_{L^\infty}$ and hence on the boundary layer limit $\mu_k$, which is also a consequence of the induction, is the desired result.
\end{proof}
 
\section{Continuity Estimate for Homogenized Boundary Data Associated with Linear Systems}\label{sec: continuity}
In this section we use the limiting structure at rational directions established above to prove that the homogenized boundary condition associated with a linear system is continuous. We recall the second cell problem, let $\xi \in \integer^d \setminus \{0\}$ a rational direction and suppose that we have a sequence of directions $n_k \to \hat\xi$ such that,
 \[ \frac{\hat\xi-n_k}{|\hat \xi-n_k |} \to \eta  \ \hbox{ a unit vector with } \ \eta \perp \xi.\]
  Then the limit of $\varphi_*(n_k)$ is determined by the following second cell problem,
 \begin{equation}
\left\{
\begin{array}{ll}
 - \grad \cdot (A^0\grad w_{\xi,\eta}) = 0 & \hbox{ in } P_\xi \vspace{1.5mm}\\
 w_{\xi,\eta} = \varphi_*(\xi,x \cdot \eta) & \hbox{ on } \partial P_\xi,
\end{array}
\right. \ \hbox{ then } \ \lim_{k \to \infty} \varphi_*(n_k) = \lim_{R \to \infty} w_{\xi,\eta}(R\xi).
\end{equation}
Where $A^0$, constant, is the homogenized matrix associated with $A(\frac{\cdot}{\e})$ and $\varphi_*(\xi,\cdot)$ defined in \eref{phi*} is a $1/|\xi|$ periodic function on $\real$ (see Lemma 2.9 in \cite{Feldman:2015aa} where the period of $\varphi_*$ is explained).

First we state the qualitative result, identifying the limit and showing continuity at rational directions.  Continuity of $\varphi_*$ at the irrational directions has been established, for example in Prange \cite{Prange:2013aa}.  Combining those results shows that $\varphi_*$ extends to a continuous function on $S^{d-1}$.
\begin{lem}\label{lem rational dir limit}
Let $\xi \in \integer^d \setminus \{0\}$ then for any sequence $n_k \to \hat \xi$,
\[ \lim_{k \to \infty} \varphi_*(n_k) = |\xi| \int_0^{1/|\xi|} \varphi_*(\xi,t) \ dt.\]
\end{lem}
From this we know that $L(\xi,\eta)$, defined in Section \ref{subsect homo inter scale}, is independent of $\eta$ in the linear case. And we will simply write $L(\xi)=L(\xi,\eta)$.
\begin{proof}
By rotation and rescaling we can reduce to proving that the boundary layer limit associated with the half-space problem,
 \begin{equation}
\left\{
\begin{array}{ll}
 - \grad \cdot (A^0\grad v) = 0 & \hbox{ in } \real^d_+ \vspace{1.5mm}\\
v = g(x_1,\dots,x_{d-1}) & \hbox{ on } \partial \real^d_+,
\end{array}
\right.
\end{equation}
where $A^0$ is a constant and uniformly elliptic and $g$ is a $\integer^{d-1}$-periodic continuous function $\real^{d-1} \to \real^N$ is,
\[ \lim_{R \to \infty} v(Re_d) = \int_{[0,1]^{d-1}} g(x) \ dx.\]
Consider the (linear) map $T : C(\mathbb{T}^{d-1}) \to \real^N$ mapping $g \mapsto \lim_{R \to \infty} v(Re_d)$.  The $L^\infty$ estimates Lemma~\ref{lem: system max} imply that $T$ is continuous.  Since $A^0$ is constant translating $g$ parallel to $\partial \real^d_+$ just translates the solution $v$ and so we also get translation invariance, for any $y \in \mathbb{T}^{d-1}$,
\[ T g(\cdot - y) = Tg.\]
The Riesz representation theorem implies that $T g = \int_{\mathbb{T}^{d-1}} g(x) \ d\mu(x)$ for some (vector-valued) measure $\mu$, then by the translation invariance, uniqueness of Haar measure, and that $T 1 = 1$ we obtain the result.
\end{proof}

The next result is quantitative, the argument uses the Dirichlet approximation theorem as in \cite{Feldman:2015aa}.
\begin{thm}
Let $\varphi_*(\cdot)$ be defined the boundary layer limit associated with \eref{cellmain} defined for $n \in S^{d-1} \setminus \real \integer^d$. Then for every $\alpha < 1/d$ and all $n_1,n_2 \in S^{d-1} \setminus \real \integer^d$,
\[ |\varphi_*(n_1) - \varphi_*(n_2)| \lesssim_\alpha \|\varphi\|_{C^5}|n_1 - n_2|^{\alpha}\]
\end{thm}
\begin{proof}
Let $n_1,n_2$ be any pair of irrational unit vectors and call $\e=|n_1-n_2|$. Let $M=\e^{-\frac{s}{s+1}}$ with $s=d-1$. By Dirichlet's Approximation Theorem (Lemma 2.11 in \cite{Feldman:2015aa}), there exists $\xi \in \mathbb{Z}^{d} \setminus \{0\}$ and $k\in \mathbb{Z}$ with $1\leq k \leq M$ such that:
\[|n_1-k^{-1}\xi |\leq Ck^{-1}M^{-{1/s}}.\]
Also 
\[|n_2-k^{-1}\xi |\leq \e+Ck^{-1}M^{-{1/s}}.\]
Note $|\xi |\lesssim k$, so 
\[|\xi||\e+Ck^{-1}M^{-{1/s}}| \lesssim \e^\frac{1}{s+1}.\]
Apply Theorem \ref{mainineq}, for any $0<\alpha <1$ we have,
\begin{align*}
|\varphi_*(n_1)-\varphi_*(n_2)|&\leq |\varphi_*(n_1)-L(\xi)|+|L(\xi)-\varphi_*(n_2)| \\
&\lesssim \e^\frac{\alpha}{1+s}  = \e^\frac{\alpha}{d}.  
\end{align*}
\end{proof}

\section{A Nonlinear Equation with Discontinuous Homogenized Boundary Data}\label{sec: discontinuity}

In this final section we study the second cell problem \eref{cell2} for nonlinear equations.  We give an example of a nonlinear divergence form equation, with smooth boundary condition, for which the boundary layer limit of \eref{cell2} depends on the approach direction $\eta$.

We consider the nonlinear operator
\[a(p_1,p_2,p_3)=\left(p_1,p_2,p_3+f(p_1,p_3)\right)^t\]
where
\[
f(p_1,p_3):=\frac{1}{8}\left( \sqrt{8p_1^2+9p_3^2}+p_3\right).
\]
Here $f$ is a solution of
\[
8f^2-2p_3 f-(p_1^2+p_3^2)=0.\]
It is easy to check that $f$ is positively $1$-homogeneous and uniformly elliptic.

We will take $\xi = e_3$ and $\eta= e_1$ or $e_2$ and we will call $(x_1,x_2,x_3) = (x,y,z)$.  For the boundary condition we choose,
\[ \varphi(y) = \frac{1}{3}+ \cos(y \cdot  \xi) \ \hbox{ so that } \ \varphi_*(\xi,s)=\frac{1}{3}+\cos(s).\]
It is worthwhile to note that arbitrary $\varphi_*(\xi,s)$ can be achieved by choose $\varphi(y) = \varphi_*(\xi,y \cdot \xi)$.  We aim to compute $L(\xi,\eta)$.

If $\eta=e_1$ \eref{cell2} becomes  
\begin{equation}\label{counter1}
\left\{
\begin{array}{ll}
 -\nabla\cdot \left(u_x,u_y,u_z+f(u_x,u_z)\right)=0 & \hbox{ in } \real^3_+, \vspace{1.5mm}\\
 u(x,y,0) =\frac{1}{3}+\cos x & \hbox{ in } \real^3_+.
 \end{array}\right.
\end{equation}
The operator and boundary data were chosen to make the solution
\[u(x,y,z)=(\frac{1}{3}+\cos x)e^{-z}.\]
Note that
\[f(u_x,u_z)=\frac{1}{3}e^{-z}\]
and so
\[(u_x,u_y,u_z+f(u_x,u_z))=(-\sin x\; e^{-z},\;0,\;-\cos x\; e^{-z})\]
from which it is easy to verify that $u$ solves \eqref{counter1}. The boundary layer limit in this case is $0$ and so, by its definition, $L(\xi,e_1) = 0$.

If $\eta = e_2$ then the equation becomes 
\begin{equation}
\left\{
\begin{array}{ll}
 -\nabla\cdot \left(u_x,u_y,u_z+f(u_x,u_z)\right)=0 & \hbox{ in } \real^3_+, \vspace{1.5mm}\\
 u(x,y,0) =\frac{1}{3}+\cos y & \hbox{ in } \real^3_+.
 \end{array}\right.
\end{equation}
This reduces to the following two-dimensional problem for $v(y,z) = u(x,y,z)$
\begin{equation}\label{counter2}
\left\{
\begin{array}{ll}
 -\nabla\cdot \left(v_y,\frac{9}{8}v_z+\frac{3}{8}|v_z|\right)=0 & \hbox{ in } \real^2_+ \vspace{1.5mm}\\
 v(y,0) =\frac{1}{3}+\cos y & \hbox{ on } \partial \real^2_+.
 \end{array}\right.
\end{equation}
Let $v$ be the solution of \eqref{counter2}.  Consider $w(y,z):=\left(\frac{1}{3}+\cos y\right)e^{-z}$, the solution from before,
\begin{align}
-\nabla\cdot \left(w_y,\frac{9}{8}w_z+\frac{3}{8}|w_z|\right) &=[(-\tfrac{4}{9} - \tfrac{1}{3}\cos y){\bf 1}_{\{\cos y  < 0\}} + \tfrac{1}{4}(\cos y - 1){\bf 1}_{\{ \cos y >0\}}]e^{-z} \notag\\
\label{subsolnw} &\leq 0. 
\end{align}
Thus $w$ is a subsolution of \eqref{counter2}, from Lemma~\ref{lem: hs comparison nonlinear} we have $w \leq v$.

The operator $(v_y,\frac{9}{8}v_z+\frac{3}{8}|v_z|)$ is uniformly elliptic and lipschitz continuous. We use a strong maximum principle of Serrin~\cite{serrin1970strong} (see Theorem $1'$ there), in any bounded domain, we either have $w\equiv v$ or $w<v$. Since the inequality in \eqref{subsolnw} is strict, except when $y = 0 \bmod 2\pi$, the case must be $w < v$. Since both $w,v$ are $1$-periodic in $y$ direction, restricting to the set $z=1$, $w(y,1)\leq v(y,1)-\delta$ for some $\delta>0$. Then by comparing $w$ and $v-\delta$ on $z\geq 1$, again using Lemma~\ref{lem: hs comparison nonlinear}, we deduce that $w\leq v-\delta$, in particular 
\[\lim_{z\to\infty}v\geq \lim_{z\to\infty}w+\delta=\delta.\]
Thus $L(\xi,e_2) <0 = L(\xi,e_1)$ and thus $\varphi_*(n)$ is discontinuous at the direction $e_3$.

\appendix

\section{}\label{sec: A}

\subsection{H\"older estimate in cone domain} We complete the proof of Lemma~\ref{lem: bdry cont linear} the H\"older estimate in the flat cone domain which we used above.
 \begin{proof}[Proof of Lemma~\ref{lem: bdry cont linear}]
  Suppose that 
  \[ | \grad g\|_{L^\infty(\partial \Omega \cap B_1)} \leq 1 \ \hbox{ and } \ \dashint_{B_1 \cap \Omega} |u^\e- g(0)|^2 \leq 1.\]
    Let some $\alpha < \alpha'<1$, by Lemma~\ref{lem: flat reg} there is a $1>\theta>0$ so that if $K_{\Sigma} = P_{n}$ for some $n \in S^{d-1}$ then,
\[ \sup_{B_\theta \cap P_n} |u^\e-g(0)| \leq \theta^{\alpha'}. \]
We prove by compactness that there exists $\delta>0$ sufficiently small such that for any solution $u^\e$ as above
\begin{equation}\label{e.compact}
 \left(\dashint_{B_\theta \cap \Omega} |u^\e - g(0)|^2\right)^{1/2} \leq \theta^\alpha . 
 \end{equation}
To achieve the H\"{o}lder estimate from \eref{compact} is the standard iteration argument.

 Suppose that the previous statement fails, that is there exists $f_k$ and corresponding $\Omega_k$ with $\delta_k = \| \grad f_k\|_\infty \to 0$, $A_k$ satisfying the standard assumptions, $\e_k >0$, $g_k$ with Lipschitz norm at most $1$ and corresponding $u_k$ solving the equation with boundary data $g_k$ on $\partial \Omega_k \cap B_1$ and
\[ \left(\dashint_{B_\theta \cap \Omega_k} |u_k - g_k(0)|^2\right)^{1/2} > \theta^\alpha.\]
  By taking subsequences we can assume that $A_k \to A$ uniformly, $g_k \to g$ uniformly and the $u_k$ converge to some $u$ weakly in $H^1$ and strongly in $L^2$. Then, assuming that $\e_k \to \e >0$, we claim $u$ solves
  \begin{equation}\label{e.compactness1}
  - \grad \cdot A(\tfrac{x}{\e}) \grad u = 0 \ \hbox{ in } \Omega \cap B_1 \ \hbox{ with } \ u = g \ \hbox{ on } \ \{x_d = 0 \} \cap B_1.
  \end{equation}
If $\e_k \to 0$ or to $\infty$ then we replace $A(x/\e)$ by $A^0$ or $A(0)$ respectively.

  The only part which is not the same as in \cite{Avellaneda:1987aa} is to check the boundary condition.  Consider the transformations
  \[ \Phi_k(x) = (x',x_d + f_k(x')) \ \hbox{ mapping } \ \Phi_k : \{x_d>0\} \to \{x_d > f_k(x')\}. \]
  Define $v_k = u_k \circ \Phi_k$.  Note that $|\Phi_k - x| \leq \delta_k$, $\grad v_k = \grad \Phi_k \grad u_k$ and $\|\grad \Phi_k - I\|_{L^\infty} \leq \delta_k$.  Therefore, up to taking a subsequence, the $v_k$ converge weakly in $H^1(B_{1}^+)$ and strongly in $H^{1/2}(B_1^+)$ to the same limit $u$.  Since the trace operator is is continuous $T: H^{1/2}(B_1^+) \to L^2(\{x_d = 0\} \cap B_1)$ we have that the trace of $v$ is the limit of the traces $g_k$ of the $v_k$.

  Then, once we have established the limit \eref{compactness1}, from the regularity estimate in the flat domain
  \[\theta^\alpha \leq \left(\dashint_{B_\theta \cap P_n} |u - g(0)|^2\right)^{1/2} \leq \theta^{\alpha'}\]
  which is a contradiction since $\alpha < \alpha'$ and $\theta <1$.

  \end{proof}

\subsection{Poisson kernel bounds in half-space intersection}
We return to prove the Poisson kernel bounds in the intersection of nearby half-spaces, Lemma~\ref{lem: PK bounds 1}.

\begin{proof}[Proof of Lemma~\ref{lem: PK bounds 1}]
The proof basically follows the proof of the Poisson kernel bounds in a smooth domain in \cite{Avellaneda:1987aa} (Lemma 21) except we need to be careful to deal with the singularity of the boundary.  We do the case $d \geq 3$, the $d=2$ case is a similar modification of the arguments in \cite{Avellaneda:1987aa} (Lemma 21).  Let $x,y \in K$ and call $r = |y-x|$.  We have the Green's function bound holding for $x,y \in K$, see Theorem 13 in \cite{Avellaneda:1987aa} and the remark below,
\[ |G_K(x,y)| \lesssim \frac{1}{r^{d-2}}.\]
We will first improve the Green's function bound, the bound on the Poisson kernel will follow.

If $\delta(x) > \frac{1}{3}r$ then $|G_K(x,y)| \lesssim \frac{\delta(x)}{r^{d-1}}$. Consider the case $\delta(x) < \frac{1}{3}r$.  Let $\overline{x} \in \partial K$ with $|x-\overline{x}| = \delta(x)$.  Then $G_K(\cdot,y)$ is a solution of the system in $B(\overline{x},r/2) \cap K$. For $\e$ sufficiently small depending on $\alpha$ the boundary H\"{o}lder estimates Lemma~\ref{lem: bdry cont linear} apply and
\begin{equation}\label{e.intboundG}
 G(z,y) \lesssim \frac{\delta(z)^\alpha}{r^{d-2+\alpha}} \ \hbox{ for all } \ z \in B(\overline{x},r/3) \cap K 
 \end{equation}
in particular the bound holds at $z = x$.

Now we make a similar argument in the $y$ variable starting from \eref{intboundG}, however H\"{o}lder regularity is not sufficient anymore so we need to deal more directly with the singularity. Since we will send $y \to \partial K \setminus \{y_1 = 0\}$ we can just consider the case $\delta(y) \leq \min\{\frac{1}{3}r,|y_1|/2\} $.  Let $\overline{y} \in \partial K$ with $|y-\overline{y}| = \delta(y)$.  Then $G_K(x,\cdot)$ is a solution of the adjoint equation in $B(\overline{y},r/2) \cap K$.  If $|y_1| \geq r/2$ then $|\overline{y}_1| \geq |y_1| \geq  r/2$ and $B(\overline{y},r/2) \cap K$ is the intersection of a half space with the ball $B(\overline{y},r/2)$.  The boundary Lipschitz estimate of \cite{Avellaneda:1987aa} applies and
\[ |G(x,z)| \lesssim \frac{\delta(z)\delta(x)^\alpha}{r^{d-1+\alpha}} \ \hbox{ for all } \ z \in B(\overline{y},r/3) \cap K,\]
since $\delta(y) \leq \frac{1}{3}r$ we get the bound at $z=y$.  If $|y_1| \leq r/2$ then we instead apply the boundary Lipschitz estimate in $B(\overline{y},|y_1|)$ to find,
\[ |G(x,z)| \lesssim \frac{\delta(z)\delta(x)^\alpha}{|y_1|r^{d-2+\alpha}} \ \hbox{ for all } \ z \in B(\overline{y},|y_1|/2) \cap K \]
since $\delta(y) \leq |y_1|/2$ we get the bound at $z=y$.  The bounds for the Poisson kernel follow by taking appropriate difference quotients.
\end{proof}

\subsection{Large scale boundary regularity nonlinear equations} We return to prove the De Giorgi boundary H\"{o}lder estimates, Lemma~\ref{lem: bdry cont nonlinear}, for scalar equations with bounded uniformly elliptic coefficients.
 \begin{proof}[Proof of Lemma~\ref{lem: bdry cont nonlinear}]
 Without loss we can assume that $\osc_{\Omega \cap B_1} u = 1$ and $0 \leq u \leq 1$ in $\Omega \cap B_1$. Call $M= \max_{\partial \Omega \cap B_1} \varphi$ and consider,
 \[ v= (u -M)_+ \ \hbox{ which is a sub-solution of } \  - \grad \cdot( A(x) \grad v )\leq 0 \ \hbox{ in } \ B_1. \]
 Now since,
 \[ |\{ v \leq 0\} \cap B_1| \geq \mu \]
 we apply the De Giorgi weak Harnack inequality to find,
 \[ v \leq (1-\delta)\left(\max_{\Omega \cap B_1} u - M\right) \ \hbox{ in } \ B_{1/2},\]
 for some $\delta>0$ depending on $\mu, d,\lambda$.  Making the same argument for $-u$ we find,
 \[ \osc_{\Omega \cap B_{1/2} } u \leq (1-\delta) \osc_{\Omega \cap B_1} u +\delta \osc_{\partial \Omega \cap B_1} \varphi.\]
 Iterating this argument we obtain,
 \[ \osc_{\Omega \cap B_{1/2^k} } u \leq (1-\delta)^k \osc_{\Omega \cap B_1} u + \sum_{j=0}^{k-1}\delta(1-\delta)^{k-j-1} \osc_{\partial \Omega \cap B_{1/2^{j}}} \varphi.\]
 Using the H\"{o}lder continuity of $\varphi$,
 \[ \osc_{\Omega \cap B_{1/2^k} } u \leq (1-\delta)^k\left(\osc_{\Omega \cap B_1} u +[\varphi]_{C^\beta}\sum_{j=0}^{k-1}\delta(1-\delta)^{-j-1} 2^{-j\beta} \right)\]
 Choosing $\delta$ smaller if necessary so that,
 \[ 2^{-\beta} < (1-\delta),\]
 the summation is bounded independent of $k$ and,
 \[ \osc_{\Omega \cap B_{1/2^k} } u \leq C(\alpha) \left(\osc_{\Omega \cap B_1} u +[\varphi]_{C^\beta}\right) 2^{-\alpha k} \ \hbox{ with } \ \alpha = -\frac{\log (1-\delta)}{\log2} < \beta.\]
 \end{proof}

\section{}\label{sec: B}

In this section we complete the proof of Lemma~\ref{lem: bdry layer rhs}.  Recall that we are considering the boundary layer problem with,
 \begin{equation}\label{e.appendixeqn}
  \left\{
 \begin{array}{ll}
 - \grad \cdot(A(x) \grad v) = \grad \cdot f & \hbox{ in } \real^d_+ \vspace{1.5mm}\\
 v = \psi(x') & \hbox{ on } \partial \real^d_+.
 \end{array}
 \right.
 \end{equation}
 where $\psi : \partial \real^d_+ \to \real$ and $f$ are smooth, $A$ satisfies the usual assumptions from Section~\ref{sec: linear background} and, furthermore, $\psi$, $f$, and $A$ all share $d-1$ linearly independent periods $\ell_1,\dots, \ell_{d-1} \in \partial \real^d_+$ such that, for some $M >2$,
 \[ \max_{1 \leq j \leq d-1} |\ell_j| \leq M.\]
 The following maximal function type norms turn out to be useful,
   \begin{equation}\label{e.fnormdef}
M_p(f,R) := \sup_{y \cdot e_d = 0, R' \geq R} \|f\|_{L^p_{avg}(B_{R'/2}(y+R'e_d))} 
  \end{equation}
  and
  \begin{equation}
  I_p(f) := M_p(f,0) + \sum_{N \in 2^{\mathbb{N}}} N M_p(f,N).
  \end{equation}
 Note that $M_p(f,0) = \|f\|_{L^\infty(\real^d_+)}$.

  We write $v$ by the Green's function formula,
  \[ v(x) = \int_{\partial \real^d_+} P(x,y) \psi(y) \ dy + \int_{\real^d_+} \grad_x G(x,y) f(y) \ dy.\]
The first result is an $L^\infty$ estimate,
\begin{lem}\label{lem: rhs infty}
For any $p>d$,
\[ \osc_{\real^d_+} v \lesssim_p \osc_{\partial \real^d_+} \psi + I_p(f) \]
\end{lem}
\begin{proof}
The bound for the Poisson integral is already done in Lemma~\ref{lem: system max}.  For the Green's function term we use the Avellaneda-Lin bounds in Lemma~\ref{integral} along with a Whitney type decomposition.  

Let $x \in \real^d_+$, without loss $x = (0,x_d)$.  If $x_d \geq 1$ call $N_x \in 2^{\mathbb{N}}$ to be the unique dyadic such that $N_x \leq x_d < 2N_x$.  Then define $\alpha \in [1,2)$ such that $\alpha N_x = x_d$.  If $x_d \leq 1$ define $\alpha = 2$.  Now we make a cube decomposition, $Q_{N,j} := \alpha N (j,1) + \alpha [ -N/4,N/2)^d$ for $2 \leq N \in 2^{\mathbb{N}}$ and $j \in \integer^{d-1}$, with side length comparable to distance to $x_d = 0$.  For $N = 1$ we define $Q_{1,j} = \alpha (j,1) + \alpha [ -1,1/2)^d$. In this set up $(0,x_d) \in Q_{N_x,0}$.

Now we bound the Green's function integral by,
\begin{align*}
 \int_{\real^d_+} |\grad_x G(x,y)| |f(y)| \ dy   &= \sum_{Q}  \int_{Q} |\grad_x G(x,y)| |f(y)| \ dy \\
 & \leq \sum_{Q} |Q|\|\grad_x G(x,\cdot)\|_{L^p_{avg}(Q)} \|f \|_{L^p_{avg}(Q)} \\
 & \leq \sum_{N} \sum_{j} N^{d}\|\grad_x G(x,\cdot)\|_{L^{p'}_{avg}(Q)} M_p(f,N).
 \end{align*}
 We claim that, for any $p>d$ and any $N \in 2^{\mathbb{N}}$, $j \in \mathbb{Z}^{d-1}$,
 \begin{equation}\label{e.whitneybound}
 \|\grad_x G(x,\cdot)\|_{L^{p'}_{avg}(Q)} \lesssim_p N^{1-d}(1+|j|)^{-d}. 
 \end{equation}
 Taking the bound for granted we can complete the computation,
\[ \int_{\real^d_+} |\grad_x G(x,y)| |f(y)| \ dy \lesssim  \sum_{N} \sum_{j} N (1+|j|)^{-d} M_p(f,N)  \lesssim I_p(f),\]
where for the last inequality we used that $(1+|j|)^{-d}$ is summable on $\integer^{d-1}$.

Now we finish by proving \eref{whitneybound} using the Avellaneda-Lin bounds, Lemma~\ref{integral}.  When $j=0$ and $N = N_x$ we bound,
\begin{align*}
 \left(\frac{1}{|Q|}\int_Q |\grad_x G(x,y)|^{p'} \ dy \right)^{1/p'} &\lesssim \left(\frac{1}{|Q|}\int_Q |x-y|^{(1-d)p'} \ dy \right)^{1/p'} \\
 &\lesssim N^{-d/p'}\left(\int_0^{CN} r^{(1-d)p'}r^{d-1} \ dr\right)^{1/p'}\\
 & \lesssim N^{-d/p'}N^{((1-p')(d-1)+1)/p'} = N^{1-d}
 \end{align*}
 where we have used $p>d$ so that $(p'-1)(d-1)<1$ and the integral in the second line converges.
 
 When $j \neq 0$ and/or $N \neq N_x$ we have that $|x-y| \gtrsim \max\{N(1+|j|),N_x\}$ for $y \in Q_{N,j}$.  In this case,
 \begin{align*}
 |\grad_x G(x,y)| & \lesssim \frac{y_d}{|x-y|^d}+\frac{x_dy_d}{|x-y|^{d+1}}\\
 &\lesssim N^{1-d}(1+|j|)^{-d} + NN_x\max\{N(1+|j|),N_x\}^{-(d+1)} \\
 & \lesssim N^{1-d}(1+|j|)^{-d},
 \end{align*}
 which was the desired estimate.
\end{proof}

 Next we prove the existence of a boundary layer limit with convergence rate.  We assume the following exponential type bounds on $f$, which are well suited to the boundary layer problem, there are $K,b > 0 $ so that, for all $R>0$,
  \begin{equation}\label{e.fnorm'}
M_p(f,R) \leq \frac{K}{1+R} e^{-bR/M}.
  \end{equation}
  From \eref{fnorm'} one can compute,
  \[ I_p(f) \lesssim_b K\log M,\]
 and also
  \[ I_p(f,R) := \sum_{2^{\mathbb{N}} \ni N \geq R} NM_p(f,N) \lesssim_b K\log Me^{-bR/M}  \]

\begin{lem}\label{lem: layer limit rhs}
Let $v$, $f$, $\psi$ and $A$ as above in \eref{appendixeqn} with $f$ satisfying the exponential bound \eref{fnorm'}.  There exists $c_* \in \real^m$ such that,
\[ \sup_{y \cdot e_d \geq R}|v(y) - c_*| \lesssim_b ((\osc \psi) +K \log M)e^{-c_0R/M}   \]
where the rate $c_0$ depends on $b$ and universal constants.
\end{lem}
The proof is almost the same as Lemma A.4 of \cite{Feldman:2015aa} so we omit it.

\medskip

\noindent \textbf{Acknowledgements.} The first author was partially supported by the National Science Foundation RTG grant DMS-1246999.  The second author was partially supported by National Science Foundation grant DMS-1566578.


\bigskip

%


\end{document}